\documentclass[12pt]{article}
\usepackage{amsmath,amssymb,amsthm}
\usepackage[frenchb,english]{babel}

\usepackage[latin1]{inputenc}
\usepackage{graphics}
\usepackage{graphicx}
\usepackage{epsfig}
\usepackage{amsfonts}
\usepackage{amsmath}
\usepackage{latexsym}
\usepackage{amscd}
\usepackage{color}
\usepackage{shadow}
\usepackage{a4wide}
\usepackage[all]{xy}
\usepackage{endnotes}
\usepackage{url}

\theoremstyle{plain}
\newtheorem{thm}{Theorem}[section]
\newtheorem{prop}[thm]{Proposition}
\newtheorem{cor}[thm]{Corollary}
\newtheorem{lemma}[thm]{Lemma}
\numberwithin{equation}{section}

\theoremstyle{remark}

\newtheorem{rmq}[thm]{Remark}

\newtheorem{dfn}[thm]{Definition}

\begin{document}

\title{Counterexamples to Strichartz estimates for the wave equation in domains}
\author{Oana Ivanovici  
%\footnote{The author was supported by grant A.N.R.-07-BLAN-0250}
 \\Universite Paris-Sud, Orsay,\\
Mathematiques, Bat. 430, 91405 Orsay Cedex, France\\
oana.ivanovici@math.u-psud.fr}
%\footnotetext{Universit{\'e} Paris Sud, Math{\'e}matiques, B$\hat{a}$t.430,
%91405 Orsay Cedex, France.\\ E-mail \tt{oana.ivanovici@math.u-psud.fr}.}
%\footnotetext{The author was supported by grant A.N.R.-07-BLAN-0250}
\date{ }
\maketitle

\section{Introduction}
Let $\Omega$ be the upper half plane $\{(x,y)\in\mathbb{R}^{2},x> 0,y\in\mathbb{R}\}$. Define the Laplacian on $\Omega$ to be $\Delta_{D}=\partial^{2}_{x}+(1+x)\partial^{2}_{y}$, together with Dirichlet boundary conditions on $\partial\Omega$: one may easily see that $\Omega$, with the metric inherited from $\Delta_{D}$, is a strictly convex domain. We shall prove that, in such a domain $\Omega$, Strichartz estimates for the wave equation suffer losses when compared to the usual case $\Omega=\mathbb{R}^{2}$, at least for a subset of the usual range of indices. Our construction is microlocal in nature; in \cite{doi} we prove that the same result holds true for any regular domain $\Omega\subset\mathbb{R}^{d}$, $d=2,3,4$, provided there exists a point in $T^{*}\partial\Omega$ where the boundary is microlocally strictly convex. 
\begin{dfn}
Let $q,r\geq
2$, $(q,r,\alpha)\neq(2,\infty,1)$. A pair $(q,r)$ is called $\alpha$-admissible if
\begin{equation}\label{admis}
\frac{1}{q}+\frac{\alpha}{r}\leq\frac{\alpha}{2},
\end{equation}
 and sharp $\alpha$-admissible whenever equality holds in \eqref{admis}. For a given dimension $d$,  a pair $(q,r)$ will be \emph{wave-admissible} if $d\geq 2$ and $(q,r)$ is
$\frac{d-1}{2}$-admissible; it will be \emph{Schrödinger-admissible} if
$d\geq 1$ and $(q,r)$ is sharp $\frac{d}{2}$-admissible. Finally, notice that  the endpoint
$(2,\frac{2\alpha}{\alpha-1})$ is sharp $\alpha$-admissible when $\alpha>1$.

When $\alpha=1$ the endpoint pairs are inadmissible and the endpoint estimates
for wave equation ($d=3$) and Schr\"{o}dinger equation ($d=2$) are
known to fail: one obtains a logarithmic loss of derivatives
which gives Strichartz estimates with $\epsilon$ losses.
\end{dfn}
Our main result reads as follows:
\begin{thm}\label{thm1}
Let $(q,r)$ be a sharp wave-admissible pair in dimension $d=2$ with $4<r<\infty$.
There exist $\psi_{j}\in
C^{\infty}_{0}(\mathbb{R})$ and for every small $\epsilon>0$ there exist $c_{\epsilon}>0$ and sequences
$V_{h,j,\epsilon}\in C^{\infty}(\Omega)$, $j=\overline{0,1}$ with $\psi_{j}(hD_{y})V_{h,j,\epsilon}=V_{h,j,\epsilon}$ (meaning that $V_{h,j,\epsilon}$ are localized at frequency $1/h$), such that the solution $V_{h,\epsilon}$ to the wave equation with
Dirichlet boundary conditions
\begin{equation}\label{wavvv}
(\partial^{2}_{t}-\Delta_{D})V_{h,\epsilon}=0,\quad V_{h,\epsilon}|_{[0,1]\times\partial\Omega}=0,
\quad V_{h,\epsilon}|_{t=0}=V_{h,0,\epsilon},\quad
\partial_{t}V_{h,\epsilon}|_{t=0}=V_{h,1,\epsilon}
\end{equation}
satisfies
\begin{equation}\label{ince}
\sup_{h,\epsilon>0}(\|V_{h,0,\epsilon}\|_{\dot{H}^{2(\frac{1}{2}-\frac{1}{r})-\frac{1}{q}+\frac{1}{6}(\frac{1}{4}-\frac{1}{r})-2\epsilon}(\Omega)}+\|V_{h,1,\epsilon}\|_{\dot{H}^{2(\frac{1}{2}-\frac{1}{r})-\frac{1}{q}+\frac{1}{6}(\frac{1}{4}-\frac{1}{r})-2\epsilon-1}(\Omega)})\leq 1
\end{equation}
and
\begin{equation}
\lim_{h\rightarrow 0}\|V_{h,\epsilon}\|_{L^{q}_{t}([0,1],L^{r}(\Omega))}=\infty.
\end{equation}
Moreover $V_{h,\epsilon}$ has compact support for $x$ in $(0,h^{\frac{1-\epsilon}{2}}]$ and is well localized at spatial frequency $1/h$; hence, the left hand side in \eqref{ince} is equivalent to
\[
h^{-\frac{3}{2}(\frac{1}{2}-\frac{1}{r})-\frac{1}{6}(\frac{1}{4}-\frac{1}{r})+2\epsilon}(\|V_{h,0,\epsilon}\|_{L^{2}(\Omega)}+h\|V_{h,1,\epsilon}\|_{L^{2}(\Omega)}).
\]
\end{thm}
\begin{rmq}
In this paper we are rather interested in negative results: Theorem \ref{thm1} shows that for $r>4$ losses of derivatives are unavoidable for
Strichartz estimates, and more specifically a regularity loss of at least
$\frac{1}{6}(\frac{1}{4}-\frac{1}{r})$ occurs when compared to the free case. 
\end{rmq}
\begin{rmq}
The key feature of the domain leading to the counterexample is the strict-convexity of the boundary, i.e. the presence of gliding rays, or highly-multiply reflected geodesics. The particular manifold studied in this paper is one for which the eigenmodes are explicitly in terms of Airy's functions and the phases for the oscillatory integrals to be evaluated have precise form. In a forthcoming work \cite{doi} we construct examples for general manifolds with a gliding ray, but the heart of the matter is well illustrated by this particular example which generalizes using Melrose's equivalence of glancing hypersurfaces theorem.
\end{rmq}
We now recall known results in $\mathbb{R}^{d}$. Let $\Delta_{d}$ denote the Laplace operator in the flat space $\mathbb{R}^{d}$.
Strichartz estimates read as follows (see \cite{kt98}):
\begin{prop}\label{propund}
Let $d\geq 2$, $(q,r)$ be wave-admissible and consider $u$, solution to the wave equation
\begin{equation}
(\partial^{2}_{t}-\Delta_{d})u(t,x)=0,\quad (t,x)\in
\mathbb{R}\times\mathbb{R}^{d},\quad u|_{t=0}=u_{0},\quad
\partial_{t}u|_{t=0}=u_{1}
\end{equation}
for $u_{0}, u_{1}\in C^{\infty}(\mathbb{R}^{d})$; then there is a constant $C$ such that
\begin{equation}\label{striw}
\|u\|_{L^{q}(\mathbb{R},L^{r}(\mathbb{R}^{d}))}\leq C(\|u_{0}\|_{\dot{H}^{d(\frac{1}{2}-\frac{1}{r})-\frac{1}{q}}(\mathbb{R}^{d})}+\|u_{1}\|_{\dot{H}^{d(\frac{1}{2}-\frac{1}{r})-\frac{1}{q}-1}(\mathbb{R}^{d})}).
\end{equation}
\end{prop}
\begin{prop}\label{propschr}
Let $d\geq 1$, $(q,r)$ be Schr$\ddot{o}$dinger-admissible pair
and $u$, solution to the Schr$\ddot{o}$dinger equation
\begin{equation}
(i\partial_{t}+\Delta_{d})u(t,x)=0,\quad
(t,x)\in\mathbb{R}\times\mathbb{R}^{d},\quad u|_{t=0}=u_{0},
\end{equation}
for $u_{0}\in C^{\infty}(\mathbb{R}^{d})$; then there is a constant $C$ such that
\begin{equation}
\|u\|_{L^{q}_{t}(\mathbb{R},L^{r}(\mathbb{R}^{d}))}\leq C\|u_{0}\|_{L^{2}(\mathbb{R}^{d})}.
\end{equation}
\end{prop}
Strichartz estimates in the context of the wave and Schr\"{o}dinger equations have a long history, beginning with Strichartz pioneering work \cite{stri77}, where he proved the particular case $q=r$ for the wave and (classical) Schr\"{o}dinger equations. This was later generalized to mixed $L^{q}_{t}L^{r}_{x}$ norms by Ginibre and Velo \cite{give85} for Schr\"{o}dinger equations, where $(q,r)$ is sharp admissible and $q>2$; the wave estimates were obtained independently by Ginibre-Velo \cite{give95} and Lindblad-Sogge \cite{ls95}, following earlier work by Kapitanski \cite{lev90}. The remaining endpoints for both equations  were finally settled by Keel and Tao \cite{kt98}.

For a manifold with smooth, strictly geodesically concave boundary, the Melrose and Taylor parametrix yields the Strichartz estimates for the wave equation with Dirichlet boundary condition (not including the endpoints) as shown in the paper of Smith and Sogge \cite{smso95}. If the concavity assumption is removed, however, the presence of multiply reflecting geodesic and their limits, gliding rays, prevent the construction of a similar parametrix!

%For a bounded domain in the plane, Strichartz estimates with "loss of $\frac{1}{6}$" derivatives for $r=8$ is a consequence of the results of Smith and Sogge \cite{smso06} on the optimality of $L^{8}(\Omega)$ norm bounds for spectral clusters. 
In \cite{kosmta}, Koch, Smith and Tataru obtained "log-loss" estimates for the spectral clusters on compact manifolds without boundary.
Recently, Burq, Lebeau and Planchon \cite{bulepl07} established Strichartz type inequalities on a manifold with boundary using the $L^{r}(\Omega)$ estimates for the spectral projectors obtained by Smith and Sogge \cite{smso06}. The range of indices $(q,r)$ that can be obtained in this manner, however, is restricted by the allowed range of $r$ in the squarefunction estimate for the wave equation, which control the norm of $u$ in the space $L^{r}(\Omega,L^{2}(-T,T))$, $T>0$ (see \cite{smso06}). In dimension $3$, for example, this restricts the indices to $q,r \geq 5$. The work of Blair, Smith and Sogge \cite{blsmso08} expands the range of indices $q$ and $r$ obtained in \cite{bulepl07}: specifically, they show that if $\Omega$ is a compact manifold with boundary and $(q,r,\beta)$ is a triple satisfying 
\[
\frac{1}{q}+\frac{d}{r}=\frac{d}{2}-\beta,
\]
together with the restriction 
\[
\left\{
      \begin{array}{ll}
      \frac{3}{q}+\frac{d-1}{r}\leq\frac{d-1}{2},\quad d\leq 4\\
      \frac{1}{q}+\frac{1}{r}\leq\frac{1}{2},\quad d\geq 4,
      \end{array}
      \right.
\]
then the Strichartz estimates \eqref{strichartz} hold true for solutions $u$ to the wave equation \eqref{wavvv} satisfying Dirichlet or Neumann homogeneous boundary conditions, with a constant $C$ depending on $\Omega$ and $T$. 
\begin{rmq}
Notice that Theorem \ref{thm1} states for instance that the scale-invariant  Strichartz estimates fail for $\frac{3}{q}+\frac{1}{r}>\frac{15}{24}$, whereas the result of Blair, Smith and Sogge states that such estimates hold if $\frac{3}{q}+\frac{1}{r}\leq\frac{1}{2}$. Of course, the counterexample places a lower bound on the loss for such indices $(q,r)$, and the work \cite{blsmso08} would place some upper bounds, but this concise statement shows one explicit gap in our knowledge that remains to be filled.
\end{rmq}
A very interesting and natural question would be to determine the sharp range of exponents for the Strichartz estimates in any dimension $d\geq 2$!

A classical way to prove Strichartz inequalities is to use
dispersive estimates (see \eqref{dispersion}). The fact that weakened dispersive estimates can \emph{still} imply optimal (and scale invariant) Strichartz estimates for the solution of the wave equation was first noticed by Lebeau: in \cite{gle06} he proved dispersive estimates with losses (which turned out to be optimal) for the wave equation inside a strictly convex domain from which he deduced Strichartz type estimates without losses but for indices $(q,r)$ satisfying \eqref{admis} with $\alpha=\frac{1}{4}$ in dimension $2$. 

A natural strategy for proving Theorem \ref{thm1} would be to use the Rayleigh whispering gallery modes which accumulate their energy near the boundary,
contributing to large $L^{r}$ norms.
Applying the semi-classical Schr\"odinger evolution shows that a loss
of $\frac{1}{6}(\frac{1}{2}-\frac{1}{r})$ derivatives is necessary
for the Strichartz estimates. However, when dealing with the wave
operator this strategy fails as the gallery modes satisfy the
Strichartz estimates of the free space:
\begin{thm}\label{thm2}
Let $d\geq 2$ and let $\Delta_{d-1}$ denote the Laplace operator in $\mathbb{R}^{d-1}$. Let
\begin{equation}\label{lapla}
\Delta_{D}=\partial^{2}_{x}+(1+x)\Delta_{d-1}, \quad  \text{where} \quad \Delta_{d-1}=\sum_{j=1}^{d-1}\partial^{2}_{y_{j}}.
\end{equation}
Let $\psi\in C^{\infty}_{0}(\mathbb{R}^{d-1}\setminus\{0\})$, $k\geq 1$ and $u_{0}\in E_{k}(\Omega)$, where $E_{k}(\Omega)$ is to be later defined by \eqref{eka}. 
\begin{enumerate}

\item Let $(q,r)$ be a Schr\"odinger-admissible pair in dimension $d$ with $q>2$ and consider the
semi-classical Schrödinger equation with Dirichlet boundary
condition
\begin{equation}\label{schrrr}
(\frac{h}{i}\partial_{t}-h^{2}\Delta_{D})u=0,\quad
u|_{\partial\Omega}=0,\quad u|_{t=0}=\psi(hD_{y})u_{0}.
\end{equation}
Then $u$ satisfies the following Strichartz estimates with a loss,
\begin{equation}\label{boschr}
\|u\|_{L^{q}([0,T_{0}],L^{r}(\Omega))}\lesssim
h^{-(\frac{d}{2}+\frac{1}{6})(\frac{1}{2}-\frac{1}{r})}\|u|_{t=0}\|_{L^{2}(\Omega)}.
\end{equation}
Moreover, the bounds \eqref{boschr} are optimal.

\item Let $(q,r)$ be a wave-admissible pair in dimension $d$ with $q>2$ and consider the wave equation with Dirichlet boundary conditions
\begin{equation}\label{eqond}
(\partial^{2}_{t}-\Delta_{D})u=0,\quad
u|_{\partial\Omega}=0,\quad u|_{t=0}=\psi(hD_{y})u_{0},\quad \partial_{t}u|_{t=0}=0.
\end{equation}
Then the solution $u$ of \eqref{eqond} satisfies
\begin{equation}
\|u\|_{L^{q}([0,T_{0}],L^{r}(\Omega)}\lesssim
h^{-d(\frac{1}{2}-\frac{1}{r})+\frac{1}{q}}\|u|_{t=0}\|_{L^{2}(\Omega)}.
\end{equation}
\end{enumerate}
\end{thm}
\begin{rmq}
We prove Theorem \ref{thm2} for the model case of the half-space 
\[
\Omega=\{x>0,\quad y\in\mathbb{R}^{d-1}\} 
\]
where the Laplacian $\Delta_{D}$ was defined by \eqref{lapla}.
It is very likely that,
using the parametrix introduced by Eskin \cite{esk77}, we could
obtain the same result for general operators.

Notice that if the initial data $u_{0}$ belongs to $E_{k}(\Omega)$ for some $k\geq 1$ then the solution $u(t,x,y)$ to \eqref{schrrr} localized in frequency at the level $1/h$ is given by
\[
u(t,x,y)=\frac{1}{(2\pi h)^{d-1}}\int e^{\frac{iy\eta}{h}}\hat{u}(t,x,\eta/h)d\eta,
\]
therefor
\[
\hat{u}(t,x,\eta/h)=e^{ith\lambda_{k}(\eta/h)}\psi(\eta)\hat{u}_{0}(x,\eta/h),
\]
where $\lambda_{k}(\eta)=|\eta|^{2}+\omega_{k}|\eta|^{4/3}$ and $\hat{u}_{0}(x,\eta/h)=Ai(|\eta|^{2/3}x/h^{2/3}-\omega_{k})$ is the eigenfunction of $-\Delta_{D,\eta}=-\partial^{2}_{x}+(1+x)\eta^{2}$ corresponding to the eigenvalue $\lambda_{k}$.
\end{rmq}
Theorem \ref{thm2} shows that the method we used for the Schr\"{o}dinger equation cannot yield Theorem \ref{thm1}. We
will proceed in a different manner, using co-normal waves with
multiply reflected cusps at the boundary (see Figure \ref{fig}). 

The paper is organized as follows: in Section \ref{galler} we will use gallery modes in order to prove Theorem \ref{thm2}; in
Section \ref{cuspd2} we prove Theorem \ref{thm1}. Finally, the Appendix collects several useful results.

\section*{Acknowledgements} The author would like to thank Gilles Lebeau who gave the initial idea of this work and guided her from idea to achievement and Nicolas Burq for helpful conversations. The author is also indebted to the referees for their remarks.
The author was supported by the A.N.R. grant 07-BLAN-0250.

\section{Whispering gallery modes}\label{galler}
\subsection{Strichartz inequalities}
Let $n\geq 2$, $0<T_{0}<\infty$, $\psi(\xi)\in
C^{\infty}_{0}(\mathbb{R}^{n}\setminus\{0\})$ and let
$G:\mathbb{R}^{n}\rightarrow\mathbb{R}$ be a smooth function $G\in C^{\infty}$ near
the support of $\psi$. Let  $u_{0}\in L^{2}(\mathbb{R}^{n})$ and $h\in (0,1]$ and consider the following
semi-classical problem
\begin{equation}\label{pb}
ih\partial_{t}u-G(\frac{h}{i}D)u=0,\quad
u|_{t=0}=\psi(hD)u_{0}.
\end{equation}
If we denote by $e^{-\frac{it}{h}G}$ the linear flow, the
solution of \eqref{pb} writes
\begin{equation}\label{solut}
e^{-\frac{it}{h}G}\psi(hD)u_{0}(x)=\frac{1}{(2\pi
h)^{n}}\int
e^{\frac{i}{h}(<x,\xi>-tG(\xi))}\psi(\xi)\hat{u_{0}}(\frac{\xi}{h})d\xi.
\end{equation}
Let $q\in(2,\infty]$, $r\in[2,\infty]$ and set
\begin{equation}\label{bet}
\frac{1}{q}=\alpha(\frac{1}{2}-\frac{1}{r}),\quad
\beta=(n-\alpha)(\frac{1}{2}-\frac{1}{r}).
\end{equation}
\begin{rmq}
Notice that the pair $(q,r)$ is Schr$\ddot{o}$dinger-admissible in dimension $n$ if $\alpha=\frac{n}{2}$ and wave admissible if $\alpha=\frac{n-1}{2}$.
\end{rmq}
With the notations in \eqref{bet} the Strichartz inequalities for \eqref{pb} read as follows
\begin{equation}\label{strichartz}
h^{\beta}\|e^{-\frac{it}{h}G}\psi(hD)u_{0}\|_{L^{q}((0,T_{0}],L^{r}(\mathbb{R}^{n}))}\leq C
\|\psi(hD)u_{0}\|_{L^{2}(\mathbb{R}^{n})}.
\end{equation}
The classical way to prove \eqref{strichartz} is to use dispersive
inequalities which read as follows
\begin{equation}\label{dispersion}
\|e^{-\frac{it}{h}G}\psi(hD)u_{0}\|_{L^{\infty}(\mathbb{R}^{n})}\lesssim (2\pi
h)^{-n}\gamma_{n,h}(\frac{t}{h})\|\psi(hD)u_{0}\|_{L^{1}(\mathbb{R}^{n})}
\end{equation}
for $t\in [0,T_{0}]$, where we set
\begin{equation}\label{gama}
\gamma_{n,h}(\lambda)=\sup_{z\in\mathbb{R}^{n}}|\int
e^{i\lambda(z\xi-G(\xi))}\psi(\xi)d\xi|.
\end{equation}
In Section \ref{ttstar} of the Appendix we prove the following:
\begin{lemma}\label{lem1}
Let $\alpha\geq 0$ and $(q,r)$ be an $\alpha$-admissible pair in dimension $n$ with $q>2$. Let $\beta$ be given by \eqref{bet}. If the solution $e^{-\frac{it}{h}G}(\psi(hD)u_{0})$ of \eqref{pb} satisfies the dispersive estimates \eqref{dispersion} for some function
$\gamma_{n,h}:\mathbb{R}\rightarrow\mathbb{R_{+}}$, then there exists some $C>0$ independent of $h$ such that the following
inequality holds
\begin{equation}\label{estim}
h^{\beta}\|e^{-\frac{it}{h}G}\psi(hD)u_{0}\|_{L^{q}((0,T_{0}],L^{r}(\mathbb{R}^{n}))}\leq C\Big(\sup_{s\in
(0,\frac{T_{0}}{h})}
s^{\alpha}\gamma_{n,h}(s))\Big)^{\frac{1}{2}-\frac{1}{r}}\|u_{0}\|_{L^{2}(\mathbb{R}^{n})}.
\end{equation}
\end{lemma}

\subsection{Gallery modes}
Let $\Omega=\{(x,y)\in\mathbb{R}^{d}| x>0,y\in\mathbb{R}^{d-1}\}$ denote the half-space $\mathbb{R}^{d}_{+}$ with the Laplacian given by \eqref{lapla} with Dirichlet boundary
condition on $\partial\Omega$. Taking the Fourier transform in the $y$-variable gives
\begin{equation}
-\Delta_{D,\eta}=-\partial^{2}_{x}+(1+x)|\eta|^{2}.
\end{equation} 
For $\eta\neq 0$, $-\Delta_{D,\eta}$ is a self-adjoint,
positive operator on $L^{2}(\mathbb{R}_{+})$ with compact resolvent. Indeed, the potential $V(x,\eta)=(1+x)\eta^{2}$ is bounded from below, it is continuous and $\lim_{x\rightarrow\infty}V(x,\eta)=\infty$. Thus one can consider the form associated to $-\partial^{2}_{x}+V(x,\eta)$,
\[
Q(u)=\int_{x>0}|\partial_{x} v|^{2}+V(x,\eta)|v|^{2}dx,\quad D(Q)=H^{1}_{0}(\mathbb{R}_{+})\cap\{v\in L^{2}(\mathbb{R}_{+}), (1+x)^{1/2}v\in L^{2}(\mathbb{R}_{+}))\},
\]
which is clearly symmetric, closed and bounded from below. If $c\gg 1$ is chosen such that $-\Delta_{D,\eta}+c$ is invertible, then $(-\Delta_{D,\eta}+c)^{-1}$ sends $L^{2}(\mathbb{R}_{+})$ in $D(Q)$ and we deduce that $(-\Delta_{D,\eta}+c)^{-1}$ is also a (self-adjoint) compact operator. The last assertion follows from the compact inclusion
\[
D(Q)=\{v| \partial_{x}v, (1+x)^{1/2}v\in L^{2}(\mathbb{R}_{+}), v(0)=0\}\hookrightarrow L^{2}(\mathbb{R}_{+}).
\]
We deduce that there exists a base of eigenfunctions $v_{k}$ of $-\Delta_{D,\eta}$ associated to a sequence of eigenvalues $\lambda_{k}(\eta)\rightarrow\infty$. From $-\Delta_{D,\eta}v=\lambda v$ we obtain $\partial^{2}_{x}v=(\eta^{2}-\lambda+x\eta^{2})v$, $v(0,\eta)=0$ and after a change of variables we find the
eigenfunctions 
\begin{equation}
v_{k}(x,\eta)=Ai(|\eta|^{\frac{2}{3}}x-\omega_{k}),
\end{equation} 
where $(-\omega_{k})_{k}$ are the zeros of Airy's
function in decreasing order. 
The corresponding  eigenvalues are $\lambda_{k}(\eta)=|\eta|^{2}+\omega_{k}|\eta|^{\frac{4}{3}}$.
\begin{dfn}
For $x>0$ let $E_{k}(\Omega)$ be the closure in $L^{2}(\Omega)$ of
\begin{equation}\label{eka}
\{u(x,y)=\frac{1}{(2\pi)^{d-1}}\int e^{iy\eta}
Ai(|\eta|^{\frac{2}{3}}x-\omega_{k})\hat{\varphi}(\eta)d\eta, \varphi\in \mathcal{S}(\mathbb{R}^{d-1})\},
\end{equation}
where $\mathcal{S}(\mathbb{R}^{d-1})$ is the Schwartz space of rapidly decreasing functions,
\[\mathcal{S}(\mathbb{R}^{d-1})=\{f\in C^{\infty}(\mathbb{R}^{d-1})| \|z^{\alpha} D^{\beta}f\|_{L^{\infty}(\mathbb{R}^{d-1})}<\infty\quad\forall\alpha,\beta\in \mathbb{N}^{d-1}\}.
\]  
For $k$ fixed, a function $u\in E_{k}(\Omega)$ is called whispering
gallery mode. Moreover, a function $u\in E_{k}(\Omega)$ satisfies
\begin{equation}\label{one}
(\partial^{2}_{x}+x\Delta_{d-1}-\omega_{k}|\Delta_{d-1}|^{\frac{2}{3}})u=0.
\end{equation}
\end{dfn}
\begin{rmq}
We have the decomposition 
\[
L^{2}(\Omega)=\bigoplus_{\bot}E_{k}(\Omega).
\] 
Indeed, from the discussion above one can easily see that $(E_{k}(\Omega))_{k}$ are closed, orthogonal and that $\cup_{k}E_{k}(\Omega)$ is a total family (i.e. that the vector space spanned by $\cup_{k}E_{k}(\Omega)$ is dense in $L^{2}(\Omega)$).
\end{rmq}
In Section \ref{secairy} of the Appendix we prove the following:
\begin{lemma}\label{lem2}
Let $\psi$, $\psi_{1}$, $\psi_{2}\in C^{\infty}_{0}(\mathbb{R}^{d-1}\setminus\{0\})$ be
such that $\psi_{1}\psi=\psi_{1}$ and $\psi\psi_{2}=\psi$ and let $\varphi\in C^{\infty}(\mathbb{R}^{d-1})$.  Fix $k\geq 1$ and let
$u\in E_{k}(\Omega)$ be the function associated to $\varphi$ in $E_{k}(\Omega)$. For
$r\in[1,\infty]$ there exist $C_{1}$, $C_{2}>0$ such that
\begin{equation}
C_{1}\|\psi_{1}(hD_{y})\varphi\|_{L^{r}(\mathbb{R}^{d-1})}\leq
h^{-\frac{2}{3r}}\|\psi(hD_{y})u\|_{L^{r}(\mathbb{R}_{+}\times\mathbb{R}^{d-1})}\leq C_{2}\|\psi_{2}(hD_{y})\varphi\|_{L^{r}(\mathbb{R}^{d-1})}.
\end{equation}
\end{lemma}
As a consequence of Lemma \ref{lem2} we have
\begin{cor}\label{rema}
Let  $\varphi_{0}\in\mathcal{S}(\mathbb{R}^{d-1})$, $k\geq 1$ and $u_{0}\in E_{k}(\Omega)$ be such that
\begin{equation}\label{uzero}
u_{0}(x,y)=\frac{1}{(2\pi)^{d-1}}\int e^{iy\eta}
Ai(|\eta|^{\frac{2}{3}}x-\omega_{k})\hat{\varphi_{0}}(\eta)d\eta.
\end{equation}
In order to prove Theorem \ref{thm2} we shall reduce the problem to the study of Strichartz type estimates for a problem with initial data $\varphi_{0}$. More precisely, from Lemma \ref{lem2} we immediately deduce the following
\begin{enumerate}
\item if $u$ solves \eqref{schrrr} with initial data $\psi(hD_{y})u_{0}$, where $u_{0}$ is given by \eqref{uzero} then in order to prove that one can't  do better than \eqref{boschr} it is enough to establish that the solution $\varphi$ to
\begin{equation}\label{equivs}
\frac{h}{i}\partial_{t}\varphi-h^{2}(\Delta_{d-1}-\omega_{k}|\Delta_{d-1}|^{\frac{2}{3}})\varphi=0,\quad
\varphi|_{t=0}=\psi(hD_{y})\varphi_{0}.
\end{equation}
satisfies the following Strichartz type estimates
\begin{equation}\label{equivals}
\|\varphi\|_{L^{q}([0,T_{0}],L^{r}(\mathbb{R}^{d-1}))}\leq c h^{-\frac{(d-1)}{2}
(\frac{1}{2}-\frac{1}{r})}\|\psi(hD_{y})\varphi_{0}\|_{L^{2}(\mathbb{R}^{d-1})}.
\end{equation}
\item if $u$ solves \eqref{eqond} with initial data $(\psi(hD_{y})u_{0},0)$ then in order to show that the gallery modes give rise to the same Strichartz estimates as in the free case it is sufficient to prove that the solution to
\begin{equation}\label{equivw}
\partial^{2}_{t}\varphi-(\Delta_{d-1}-\omega_{k}|\Delta_{d-1}|^{\frac{2}{3}})\varphi=0,\quad
\varphi|_{t=0}=\psi(hD_{y})\varphi_{0},\quad \partial_{t}\varphi|_{t=0}=0
\end{equation}
satisfies
\begin{equation}\label{equivalw}
\|\varphi\|_{L^{q}([0,T_{0}],L^{r}(\mathbb{R}^{d-1}))}\leq c h^{-(\frac{d}{2}-\frac{1}{6})
(\frac{1}{2}-\frac{1}{r})}\|\psi(hD_{y})\varphi_{0}\|_{L^{2}(\mathbb{R}^{d-1})}.
\end{equation}
\end{enumerate}
\end{cor}
\begin{rmq}\label{rmqq}
Notice that for $\tilde{q}\geq q>2$ and $f\in C^{\infty}([0,T])$ we have
\[
\|f\|_{L^{q}([0,T])}\lesssim\|f\|_{L^{\tilde{q}}([0,T])},
\]
thus in order to prove Theorem \ref{thm2} it suffices to prove \eqref{equivals} (respective \eqref{equivalw}) with $q$ replaced by some $\tilde{q}\geq q$.
\end{rmq}

\subsection{Proof of Theorem \ref{thm2}}
Let  $\varphi_{0}\in\mathcal{S}(\mathbb{R}^{d-1})$, $k\geq 1$, $\omega=\omega_{k}>0$ and $u_{0}\in E_{k}(\Omega)$ be such that
\begin{equation}
u_{0}(x,y)=\frac{1}{(2\pi)^{d-1}}\int e^{iy\eta}
Ai(|\eta|^{\frac{2}{3}}x-\omega_{k})\hat{\varphi_{0}}(\eta)d\eta.
\end{equation}
\begin{enumerate}
\item \textbf{Schr\"{o}dinger equation}

Let $\tilde{q}$ be given by
\begin{equation}\label{tild}
\frac{1}{\tilde{q}}=\frac{(d-1)}{2}(\frac{1}{2}-\frac{1}{r}).
\end{equation}
Let $G_{s}(\eta)=|\eta|^{2}+\omega h^{\frac{2}{3}}|\eta|^{\frac{4}{3}}$. Using Corollary \ref{rema} and Remark \ref{rmqq} we are reduced to prove \eqref{equivals}, with $q$ replaced by $\tilde{q}$, i.e. in order to prove Theorem \ref{thm2} for the Schr\"{o}dinger operator it will be enough to establish 
\begin{equation}\label{surf}
\|e^{-\frac{it}{h}G_{s}}(\psi(hD_{y})\varphi_{0})\|_{L^{\tilde{q}}([0,T],L^{r}(\mathbb{R}^{d-1}))}\leq c h^{-\frac{(d-1)}{2}
(\frac{1}{2}-\frac{1}{r})}\|\psi(hD_{y})\varphi_{0}\|_{L^{2}(\mathbb{R}^{d-1})},
\end{equation}
where
\[
e^{-\frac{it}{h}G_{s}}(\psi(hD_{y})\varphi_{0})(t,y)=\frac{1}{(2\pi h)^{d-1}}\int e^{\frac{i}{h}(<y,\eta>-t(|\eta|^{2}+\omega h^{\frac{2}{3}}|\eta|^{\frac{4}{3}}))}\psi(\eta)\hat{\varphi}_{0}(\frac{\eta}{h})d\eta.
\]
Let $\omega=\omega_{k}$ and set
\begin{equation}\label{j}
J(z,\frac{t}{h}):=\int e^{i\frac{t}{h}(<z,\eta>-G_{s}(\eta))}\psi(\eta)d\eta.
\end{equation}
Recall that $0\not\in \text{supp} (\psi)$, thus the phase function is smooth everywhere on the support of $\psi$. 
With the notations in
\eqref{gama} we have to determine $\gamma_{d-1,h}(\frac{t}{h})=\text{sup}_{z\in\mathbb{R}^{d-1}}|J(z,\frac{t}{h})|$. Note that if $\frac{|t|}{h}$ is bounded we get immediately that $|J(z,\frac{t}{h})|$ is bounded, thus we can consider
the quotient $\frac{t}{h}$ to be large. Let $\lambda=\frac{t}{h}\gg 1$ and apply the stationary phase method. There is one critical point 
\[
z(\eta)=2\eta+\frac{4}{3}\omega
h^{\frac{2}{3}}\frac{\eta}{|\eta|^{\frac{2}{3}}},
\]
non-degenerate since $G''_{s}(\eta)=2Id+ O(h^{\frac{2}{3}})$ for $\eta$ away from $0$ and $h$ small enough, and we can also write $\eta=\eta(z)$. 
We obtain
\begin{equation}\label{jjjj}
J(z,\lambda)\simeq\Big(\frac{2\pi}{\sqrt{\lambda}}\Big)^{d-1}\frac{e^{-\frac{i\pi}{4}\text{sign}
G''_{s}(\eta(z))}}{\sqrt{\det G''_{s}(\eta(z))}}e^{i\lambda
\Phi(z,\eta(z))}\sigma(z,\lambda),
\end{equation}
where 
\[
\Phi(z,\eta)=<z,\eta>-G_{s}(\eta),\quad 
\Phi(z,\eta(z))=|\eta(z)|^{2}+\frac{\omega}{3}h^{\frac{2}{3}}|\eta(z)|^{\frac{4}{3}},
\] 
\[
\sigma(z;\lambda)\simeq\sum_{k\geq0}\lambda^{-k}\sigma_{k}(z),\quad \sigma_{0}(z)=\psi(\eta(z)).
\]
From the definition \eqref{gama} we deduce that we have
$\gamma_{d-1,h}(\lambda)\simeq \lambda^{-\frac{(d-1)}{2}}$. Consequently, for $\lambda=\frac{t}{h}\gg 1$ there exists some constant $C>0$ such that the following dispersive
estimate holds
\begin{equation}\label{dis}
\|e^{-\frac{it}{h}G_{s}}\psi(hD_{y})\varphi_{0}\|_{L^{\infty}_{y}(\mathbb{R}^{d-1})}\leq
Ch^{-(d-1)}(\frac{|t|}{h})^{-\frac{(d-1)}{2}}\|\psi(hD_{y})\varphi_{0}\|_{L^{1}_{y}(\mathbb{R}^{d-1})}.
\end{equation}
Interpolation between \eqref{dis} and the energy estimate gives
\begin{equation}\label{rr'}
\|e^{-\frac{it}{h}G_{s}}\psi(hD_{y})\varphi_{0}\|_{L^{r}_{y}(\mathbb{R}^{d-1})}\leq
Ch^{-\frac{(d-1)}{2}(1-\frac{2}{r})}|t|^{-\frac{(d-1)}{2}(1-\frac{2}{r})}\|\psi(hD_{y})\varphi_{0}\|_{L^{r'}_{y}(\mathbb{R}^{d-1})}.
\end{equation}
Let $\tilde{q}'$ be such that $\frac{1}{\tilde{q}}+\frac{1}{\tilde{q}'}=1$ and let $T$ be the operator
\[
T:L^{2}(\mathbb{R}^{d-1})\rightarrow L^{\tilde{q}}([0,T_{0}],L^{r}_{y}(\mathbb{R}^{d-1}))
\]
which to a given $\varphi_{0}\in L^{2}(\mathbb{R}^{d-1})$
associates $e^{-\frac{it}{h}G_{s}}\psi(hD_{y})\varphi_{0}\in L^{\tilde{q}}([0,T_{0}],L^{r}_{y}(\mathbb{R}^{d-1}))$. Then inequality
\eqref{rr'} implies that for every $g\in
L^{\tilde{q}'}([0,T_{0}],L^{r'}_{y}(\mathbb{R}^{d-1}))$ we have
\begin{equation}
\|TT^{*}g\|_{L^{\tilde{q}}(0,T_{0}]L^{r}_{y}}=\|\int_{0}^{T}e^{-\frac{i(t-s)}{h}G_{s}}\psi\psi^{*}g(s)ds
\| _{L^{\tilde{q}}((0,T_{0}],L^{r}_{y})}\leq
\end{equation}
\[
\leq
Ch^{-\frac{(d-1)}{2}(1-\frac{2}{r})}\|\int_{0}^{T}|t-s|^{-\frac{(d-1)}{2}(1-\frac{2}{r})}
\|g(s)\|_{L^{r'}_{y}}ds\|_{L^{\tilde{q}}((0,T_{0}])}.
\]
If $\tilde{q}>2$ the application $|t|^{-\frac{2}{\tilde{q}}}\ast :
L^{\tilde{q}'}\rightarrow L^{\tilde{q}}$ is bounded by the
Hardy-Littlewood-Sobolev theorem and we deduce that the
$L^{\tilde{q}}((0,T_{0}],L^{r}_{y}(\mathbb{R}^{d-1}))$ norm of the operator $\text{TT}^{*}$ is bounded from
above by $h^{-\frac{(d-1)}{2}(1-\frac{2}{r})}$, thus the norm $\|T\|_{L^{2}\rightarrow L^{\tilde{q}}((0,T_{0}],L^{r}_{y}(\mathbb{R}^{d-1}))}$ is bounded from above by $h^{-\frac{(d-1)}{2}(\frac{1}{2}-\frac{1}{r})}$.

\begin{itemize}
\item \textbf{Optimality:}

Let $\eta_{0}\in\mathbb{R}^{d-1}\setminus\{0\}$ and for $y\in\mathbb{R}^{d-1}$ set
\begin{equation}
\varphi_{0,h}(y,\eta_{0})=h^{-(d-1)/4}e^{\frac{i\Phi_{0}(y,\eta_{0})}{h}},\quad \Phi_{0}(y,\eta_{0})=<y,\eta_{0}>+\frac{i|y|^{2}}{2}
\end{equation}
that satisfies
\[
\|\varphi_{0,h}\|_{L^{2}_{y}}=\pi^{(d-1)/4},\quad
\|\varphi_{0,h}\|_{L^{\infty}_{y}}=h^{-(d-1)/4}.
\] 
\begin{prop}
Let $T_{0}>0$ be small enough and let $t\in[-T_{0},T_{0}]$. Then the (local, holomorphic) solution $\varphi_{h}(t,y,\eta_{0})$ of \eqref{equivs} with initial data $\varphi_{0,h}(y,\eta_{0})$ has the form
\[
\varphi_{h}(t,y,\eta_{0})\simeq h^{-(d-1)/4}e^{\frac{i\Phi(t,y,\eta_{0})}{h}}\sigma(t,y,\eta_{0},h),
\]
where the phase function $\Phi(t,y,\eta_{0})$ satisfies the eikonal equation \eqref{eikon} locally in time $t$ and for $y$ in a neighborhood of $0$, and where
 $\sigma(t,y,\eta_{0},h)\simeq \sum_{k\geq 0} h^{k}\sigma_{k}(t,y,\eta_{0})$ is a classical analytic symbol (see \cite[Chps.1,9; Thm.9.1]{sjos} for the definition and for a complete proof). 
 \end{prop}
\begin{proof} 
For $t\in[-T_{0},T_{0}]$ small enough we can construct approximatively 
\[
\varphi_{h}(t,y,\eta_{0})=\exp(-i\frac{t}{h}G_{s})\varphi_{0,h}(y,\eta_{0})
\]
to be the local solution to \eqref{equivs} with initial data $\varphi_{0,h}(y,\eta_{0})$.  In order to solve explicitly 
\eqref{equivs} we use geometric optic's arguments: let first $\Phi(t,y,\eta_{0})$ be the (local) solution to the eikonal equation
\begin{equation}\label{eikon}
\left\{
                \begin{array}{ll}
                \partial_{t}\Phi+|\nabla_{y}\Phi|^{2}+\omega|\nabla_{y}\Phi|^{4/3}=0,\\
                \Phi|_{t=0}=\Phi_{0}(y,\eta_{0}).
                \end{array}
                \right.
\end{equation}
The associated complex Lagrangian manifold is given by
\[
\Lambda_{\Phi}=\{(t,y,\tau,\eta)|\tau=\partial_{t}\Phi,\eta=\nabla_{y}\Phi\}.
\]
Let $q(t,y,\tau,\eta)=\tau+|\eta|^{2}+\omega|\eta|^{4/3}$ and let $H_{q}$ denote the Hamilton field associated to $q$. Then $\Lambda_{\Phi}$ is generated by the integral curves of $H_{q}$ which satisfy
\begin{equation}
\left\{
                \begin{array}{ll}
                (\dot{t},\dot{y},\dot{\tau},\dot{\eta})=(1,\nabla_{\eta}q,0,0),\\
                (t,y,\tau,\eta)|_{0}=(0,y_{0},-|\nabla_{y}\Phi_{0}|^{2}-\omega|\nabla_{y}\Phi_{0}|^{4/3},\nabla_{y}\Phi_{0}=\eta_{0}+iy_{0}).
                \end{array}
                \right.
\end{equation}
We parametrize them by $t$ and write the solution 
\[
(y(t,y_{0},\eta_{0}),\eta(t,y_{0},\eta_{0}))=\exp{(tH_{q}(y_{0},\eta_{0}))}.
\] 
The intersection
\[
\Lambda^{\mathbb{R}}_{\Phi}:=\{\exp tH_{q}(y_{0},\eta_{0})| t\in\mathbb{R}\}\cap T^{*}\mathbb{R}^{d}\setminus\{0\}
\]
is empty unless $y_{0}=0$, since it is so at $t=0$ and since $d\exp{(tH_{q})}$ preserves the positivity of the $\mathbb{C}$-Lagrangian $\Lambda_{\Phi}$ (see Definition \ref{dfnlagrang} of the Appendix and Lemma \ref{lemlag}). Thus on the bicharacteristic starting from $y_{0}=0$ the imaginary part of the phase $\Phi(t,y(t,0,\eta_{0}),\eta_{0})$ vanishes. Moreover, the following holds:
\begin{prop}
The phase $\Phi$ satisfies, for $y$ in a neighborhood of $0$, 
\[
\Phi(t,y,\eta_{0})=\Phi(t,y(t,0,\eta_{0}),\eta_{0})+(y-y(t,0,\eta_{0}))\eta(t,0,\eta_{0})
\]
\[
+(y-y(t,0,\eta_{0}))B(t,y,\eta_{0})(y-y(t,0,\eta_{0})),
\]
where the phase $\Phi(t,y(t,0,\eta_{0}),\eta_{0})$ and its derivative $\eta(t,0,\eta_{0})$ with respect to the $y$ variable are real and the imaginary part of $B(t,y,\eta_{0})\in\mathcal{M}_{d-1}(\mathbb{C})$ is positive definite.
\end{prop}
\begin{proof}
Indeed, the initial function $\Phi_{0}$ is complex valued with Hessian $\text{Im}\nabla^{2}_{y}\Phi_{0}$ positive definite. Then it follows from \cite[Prop.21.5.9]{hormand} that the complexified tangent plane of $\Lambda_{\Phi_{0}}$ is a strictly positive Lagrangian plane (see \cite[Def.21.5.5]{hormand}). The tangent plane at $(y(t,0,\eta_{0}),\eta(t,0,\eta_{0}))$ is the image of the complexified tangent plane of $\Lambda^{\mathbb{R}}_{\Phi_{0}}$ at $(y_{0}=0,\eta_{0})$ under the complexification of a real symplectic map, hence strictly positive. More details for these arguments are given in Section \ref{propapos} of the Appendix.
\end{proof}
We look now for $\varphi_{h}(t,y,\eta_{0})$ of the form $h^{-(d-1)/4}e^{\frac{i}{h}\Phi(t,y,\eta_{0})}\sigma(t,y,\eta_{0},h)$ where $\sigma=\sum h^{k}\sigma_{k}$ must be an analytic classical symbol. Substitution in \eqref{equivs} yields the following system of transport equations 
\[
\left\{       \begin{array}{ll}
                L\sigma_{0}=0,\quad \sigma_{0}|_{t=0}=1,\\
                L\sigma_{1}+f_{1}(\sigma_{0})=0,\quad \sigma_{1}|_{t=0}=0,\\
                .\\
                .\\
                L\sigma_{k}+f_{k}(\sigma_{0},..,\sigma_{k-1})=0,\quad \sigma_{k}|_{t=0}=0,\\
                .
                \end{array}
                \right.
\]
where $L=\frac{\partial}{\partial t}+\sum_{j=1}^{d-1}q^{j}(y,\nabla_{y}\Phi)+s(y,\eta_{0})$ with $s(y,\eta_{0})$ analytic and $f_{k}(\sigma_{0},..,\sigma_{k-1})$ a linear expression with analytic coefficients of derivatives of $\sigma_{0}$,.., $\sigma_{k-1}$. It is clear that we can solve this system for $y$ in some complex domain $\mathcal{O}$, independently of $k$; in \cite[Chps. 9,10]{sjos} it is shown that in this way $\sigma$ becomes an analytic symbol there.
\end{proof}
Let us define $\underline{\sigma}_{k}(t,y,\eta_{0})=\sigma_{k}(t,y,\eta_{0})$ for $(t,y)\in(-T_{0},T_{0})\times\mathcal{O}$ and $\underline{\sigma}_{k}(t,y,\eta_{0})=0$ otherwise and let $\underline{\sigma}(t,y,\eta_{0},h):=\sum_{k\geq 0}h^{k}\underline{\sigma}_{k}(t,y,\eta_{0})$. Set also 
\[
\underline{\varphi}_{h}(t,y,\eta_{0})=h^{-(d-1)/4}e^{\frac{i}{h}\Phi(t,y,\eta_{0})}\underline{\sigma}(t,y,\eta_{0},h),
\]
thus $\underline{\varphi}_{h}$ solves \eqref{equivs} for $t\in[-T_{0},T_{0}]$ and $y\in\mathbb{R}^{d-1}$ and we can compute the $L^{r}(\mathbb{R}^{d-1})$ norm of $\underline{\varphi}_{h}(t,y,\eta_{0})$ globally in $y$:
\[
\|\underline{\varphi}_{h}(t,.,\eta_{0})\|_{L^{r}(\mathbb{R}^{d-1})}=
\]
\[
h^{-(d-1)/4}(\int e^{-\frac{r}{h}(y-y(t,0,\eta_{0}))Im B(t,y,\eta_{0})(y-y(t,0,\eta_{0}))}|\underline{\sigma}(t,y,\eta_{0},h)|^{r}dy)^{1/r}\simeq
\]
\[
h^{-(d-1)/4+(d-1)/2r}(1+O(h^{-1}))
\]
and consequently for $T_{0}$ small enough we have 
\[
\|\underline{\varphi}_{h}\|_{L^{q}((0,T_{0}],L^{r}(\mathbb{R}^{d-1}))}= h^{-\frac{d-1}{2}(\frac{1}{2}-\frac{1}{r})}(1+O(h^{-1}))
\]
and we conclude using Corollary \ref{rema}.
\end{itemize}

\item \textbf{Wave equation}
Let $(q,r)$ be a sharp wave-admissible pair in dimension $d\geq 2$, $q> 2$, and let $\tilde{q}$ be given by \eqref{tild}. Using Corollary \ref{rema} and Remark \ref{rmqq} and since $\tilde{q}\geq q$, we are reduced to prove \eqref{equivalw} with $q$ replace by $\tilde{q}$, i.e. 
\[
\|e^{i\frac{t}{h}G_{w}}(\psi(hD_{y})\varphi_{0})\|_{L^{\tilde{q}}([0,T_{0}],L^{r}(\mathbb{R}^{d-1}))}\lesssim h^{-(\frac{d}{2}-\frac{1}{6})
(\frac{1}{2}-\frac{1}{r})}\|\psi(hD_{y})\varphi_{0}\|_{L^{2}(\mathbb{R}^{d-1})},
\]
where $G_{w}(\eta)=\sqrt{|\eta|^{2}+\omega
h^{\frac{2}{3}}|\eta|^{\frac{4}{3}}}$ and
\[
e^{i\frac{t}{h}G_{w}}(\psi(hD_{y})\varphi_{0})(t,y)=\frac{1}{(2\pi h)^{d-1}}\int e^{\frac{i}{h}(<y,\eta>-t\sqrt{|\eta|^{2}+\omega h^{\frac{2}{3}}|\eta|^{\frac{4}{3}}})}\psi(\eta)\hat{\varphi}_{0}(\frac{\eta}{h})d\eta.
\]
In order to obtain dispersive estimates we need the following
\begin{prop}\label{propgama}
Let $\lambda=\frac{t}{h}$ and set as before
\begin{equation}\label{gew}
J(z,\lambda):=\int e^{i\lambda(z\eta-G_{w}(\eta))}\psi(\eta)d\eta,\quad \gamma_{d-1,h}(\lambda)=\sup_{z\in \mathbb{R}^{d-1}}|J(z,\lambda)|.
\end{equation}
Then the function $\gamma_{d-1,h}$ satisfies
\[
\gamma_{d-1,h}(\lambda)\simeq h^{-1/3}\lambda^{-(d-1)/2}.
\]
\end{prop}
We postpone the proof of Proposition \ref{propgama} for the end of this section and proceed.

\emph{End of the proof of Theorem \ref{thm2}}

From \eqref{dispersion} the dispersive estimates
read as follows
\begin{equation}\label{dis1}
\|e^{-\frac{it}{h}G_{w}}\psi(hD_{y})\varphi_{0}\|_{L^{\infty}(\mathbb{R}^{d-1})}\lesssim
h^{-d+1-\frac{1}{3}}(\frac{t}{h})^{-\frac{(d-1)}{2}}\|\psi(hD_{y})\varphi_{0}\|_{L^{1}(\mathbb{R}^{d-1})}.
\end{equation}
Lemma \ref{lem1} can be applied at this point of the proof for the $\frac{d-1}{2}$ wave-admissible pair in dimension $d-1$, $(\tilde{q},r)$, in order to obtain 
\begin{equation}
\|e^{-\frac{it}{h}G_{w}}\psi(hD_{y})\varphi_{0}\|_{L^{2}\rightarrow L^{\tilde{q}}(0,T_{0}],L^{r}(\mathbb{R}^{d-1})}\lesssim
h^{-(\frac{d}{2}-\frac{1}{6})(\frac{1}{2}-\frac{1}{r})}.
\end{equation}
We conclude using again Corollary \ref{rema}. It remains to prove Proposition \ref{propgama}.
\begin{proof}\emph{of Proposition \ref{propgama}}

As before, the case of most interest in the study of $\sup_{z\in \mathbb{R}^{d-1}}|J(z,\lambda)|$ will be the one for which $\lambda\gg 1$, since when $\frac{t}{h}$ remains bounded good estimates are found immediately.
We shall thus concentrate on the case $\lambda\gg 1$ and we apply the stationary phase lemma. Notice that on the support of $\psi$ the phase function of $J$ is smooth. On the other hand, since $\eta$ stays away from a neighborhood of $0$, the critical point of $J(z,\lambda)$ satisfies
\[
z=G'_{w}(\eta)=\frac{\eta}{|\eta|}+O(h^{\frac{2}{3}}).
\]
In order to
estimate $|J(z,\lambda)|$ it will be thus enough to localize in a $h^{2/3}$ neighborhood of $|z|=1$. We shall assume without loss of generality that $\psi$ is radial and set $\tilde{\psi}(|\eta|)=\psi(\eta)$ in which case $J(.,\lambda)$ depends also only on $|z|$ and it is enough to estimate
\begin{equation}\label{jw}
J(|z|e_{1},\lambda)=\int_{0}^{\infty}\int_{\mathbb{S}^{d-2}}e^{i\lambda(|z|\rho\theta_{1}-\sqrt{\rho^{2}+\omega h^{2/3}\rho^{4/3}})}\tilde{\psi}(\rho)\rho^{d-2}d\theta d\rho,
\end{equation}
where $e_{1}=(1,0,..,0)\in\mathbb{R}^{d-1}$ and $\mathbb{S}^{d-2}$ is the unit sphere in $\mathbb{R}^{d-1}$.
The derivative with respect to $\rho$ gives $|z|\theta_{1}=\frac{\rho}{\sqrt{\rho^{2}+\omega h^{2/3}\rho^{4/3}}}\simeq 1$ and making integrations by parts with respect to $\rho$ one sees that the contribution in the integral in $\eta$ in \eqref{gew} is $O(\lambda^{-\infty})$ if $|z|\ll 1$. Consequently, one may assume $|z|\geq c>0$. As a consequence $\theta_{1}$ can be taken close to $1$, and since on the sphere $\mathbb{S}^{d-2}$ one has $\theta_{1}=\pm\sqrt{1-\theta^{'2}}$, $\theta=(\theta_{1},\theta')\in\mathbb{R}^{d-1}$, we can introduce a cutoff function $b(\theta')$ supported near $0$ such that the right hand side in \eqref{jw} writes, modulo $O((\lambda|z|)^{-\infty})$
\[
\sum_{\pm}\int_{0}^{\infty}\int_{\theta'}e^{i\lambda(\pm\rho|z|\sqrt{1-\theta^{'2}}-\sqrt{\rho^{2}+\omega h^{2/3}\rho^{4/3}})}b(\theta')\tilde{\psi}(\rho)\rho^{d-2}d\theta' d\rho.
\]
The term corresponding to the critical point $-1$ gives a contribution $O(\lambda^{-\infty})$ in the integral with respect to $\rho$ by non-stationary phase theorem. 
Using the stationary phase theorem for the integral in $\theta'$ we find
\[
J(|z|e_{1},\lambda)\simeq\int_{0}^{\infty}e^{i\lambda(|z|\rho-\sqrt{\rho^{2}+\omega h^{2/3}\rho^{4/3}})}\tilde{\psi}(\rho)\sigma_{+}(\lambda|z|\rho)d\rho+O(\lambda^{-\infty}),
\]
where $\sigma_{+}$ is a symbol of order $-(d-2)/2$. In order to estimate this term we write its phase function as follows
\[
|z|\rho-\sqrt{\rho^{2}+\omega h^{2/3}\rho^{4/3}}=(|z|-1)\rho-(\sqrt{\rho^{2}+\omega h^{2/3}\rho^{4/3}}-\rho)
\]
and set $|z|-1=h^{2/3}x$. Hence $J(|z|e_{1},\lambda)$ can be estimated by
\begin{equation}\label{supj}
J((1+h^{2/3}x)e_{1},\lambda)\simeq\int_{0}^{\infty} e^{i\mu(\rho
x-\frac{\omega\rho^{1/3}}{1+\sqrt{1+\omega h^{2/3}\rho^{1/3}}})}\tilde{\psi}(\rho)\sigma_{+}(\lambda\rho(1+h^{2/3}x))d\rho.
\end{equation}

We distinguish two cases, weather $1\ll\lambda=t/h\lesssim h^{-2/3}$ or $1\ll\mu=\lambda h^{2/3}=h^{-1/3}t$. 
\begin{itemize}
\item In the first case $\mu\lesssim 1$  and formula \eqref{supj}
give us bounds from above for $\sup_{z}|J(z,\lambda)|$ of the form
$\lambda^{-(d-2)/2}$ (recall that for $|z|$ away from a $h^{2/3}$-neighborhood of $1$ the problem was trivial by non-stationary lemma). 

\item If $1\ll\mu=h^{-1/3}t$ and $x\neq 0$ we apply the stationary phase lemma in dimension one with
phase $\Phi(\rho)=\rho x-\frac{\omega}{2}\rho^{1/3}$ which is smooth since $\rho\neq 0$ and has one critical, non-degenerate ($\Phi''(\rho)=\frac{\omega}{9}\rho^{-5/3}\neq 0$) point satisfying
\[
\rho=(6x/\omega)^{-3/2}.
\]
For values of $x$ for which $(6x/\omega)^{-3/2}$ belongs to the support of $\tilde{\psi}$ we find
\begin{equation}
J((1+h^{2/3}x)e_{1},\lambda)\simeq C(x)\lambda^{-\frac{d-2}{2}}\mu^{-\frac{1}{2}}+O(\lambda^{-\frac{d-2}{2}}\mu^{-3/2})
\end{equation}
\[
\simeq C(x)\lambda^{-\frac{d-1}{2}}h^{-\frac{1}{3}}
+O(\lambda^{-\frac{d-2}{2}}\mu^{-3/2}),
\]
with $C(x)$ bounded and consequently we can determine $\gamma_{d-1,h}$ defined by \eqref{gama} where $n=d-1$ and $G=G_{w}$. We find
\begin{equation}
\gamma_{d-1,h}(\frac{t}{h})\simeq
h^{-\frac{1}{3}}\Big(\frac{t}{h}\Big)^{-\frac{d-1}{2}},
\end{equation}
thus the proof is complete.
\end{itemize}
\end{proof}
\end{enumerate}

\section{Conormal waves with cusp in dimension $d=2$}\label{cuspd2}
In what follows let $0<\epsilon\ll 1$ be small. 
We shall construct a sequence $W_{h,\epsilon}$ of approximate solutions of the wave equation
\begin{equation}\label{ondes}
(\partial^{2}_{t}-\Delta_{D}) V(t,x,y)=0\quad \text{for}\quad  (t,x,y)\in \mathbb{R}\times\mathbb{R}^{2},
\quad V|_{[0,1]\times\partial\Omega}=0,
\end{equation}
which contradicts the Strichartz estimates of the free space (see Proposition \ref{propund}). Using the approximate solutions $W_{h,\epsilon}$ we shall conclude Theorem \ref{thm1}  by showing that one can find (exact) solutions $V_{h,\epsilon}$ of \eqref{ondes} which provide losses of derivatives for the $L^{q}([0,1],L^{r}(\Omega))$ norms of at least $\frac{1}{6}(\frac{1}{4}-\frac{1}{r})-\epsilon$ for $r>4$ when compared to the free space, $(q,r)$ being a wave-admissible pair in dimension $2$. 

\subsection{Motivation for the choice of the approximate solution}
Let the wave operator be given by
$\square=\partial^{2}_{t}-\partial^{2}_{x}-(1+x)\partial^{2}_{y}$
and let $p(t,x,y,\tau,\xi,\eta)=\xi^{2}+(1+x)\eta^{2}-\tau^{2}$ denote its (homogeneous) symbol. The characteristic set of $\square$ is the closed conic set $\{(t,x,y,\tau,\xi,\eta)|p(t,x,y,\tau,\xi,\eta)=0\}$, denoted $\text{Char}(p)$. 
We define the semi-classical wave front set $WF_{h}(u)$ of a distribution $u$ on $\mathbb{R}^{3}$ to be the complement of the set of points $(\rho=(t,x,y),\zeta=(\tau,\xi,\eta))\in\mathbb{R}^{3}\times(\mathbb{R}^{3}\setminus 0)$  for which there exists a symbol $a(\rho,\zeta)\in\mathcal{S}(\mathbb{R}^{6})$ such that $a(\rho,\zeta)\neq 0$ and for all integers $m\geq 0$ the following holds
\[
\|a(\rho,hD_{\rho})u\|_{L^{2}}\leq c_{m}h^{m}.
\]
Let $\rho=\rho(\sigma)$, $\zeta=\zeta(\sigma)$ be a bicharacteristic of $p(\rho,\zeta)$, i.e. such that $(\rho,\zeta)$ satisfies
\begin{equation}
\frac{d\rho}{d\sigma}=\frac{\partial p}{\partial\zeta},\quad\frac{d\zeta}{d\sigma}=-\frac{\partial p}{\partial\rho},\quad
p(\rho(0),\zeta(0))=0.
\end{equation}
Assume that the interior of $\Omega$ is given by the inequality $\gamma(\rho)>0$, in this case $\gamma(\rho=(t,x,y))=x$. Then $\rho=\rho(\sigma)$, $\zeta=\zeta(\sigma)$ is tangential to $\mathbb{R}\times\partial\Omega$ if 
\begin{equation}
\gamma(\rho(0))=0,\quad \frac{d}{d\sigma}\gamma(\rho(0))=0.
\end{equation}
We say that a point $(\rho,\zeta)$ on the boundary is a gliding point if it is a tangential point and if in addition
\begin{equation}\label{hypconv}
\frac{d^{2}}{d\sigma^{2}}\gamma(\rho(0))<0.
\end{equation}
This is equivalent (see for example \cite{esk77}) to saying that $(\rho,\zeta)\in T^{*}(\mathbb{R}\times\partial\Omega)\setminus 0$ is a gliding point if
\begin{equation}\label{hyconv2}
p(\rho,\zeta)=0,\quad \{p,\gamma\}|_{(\rho,\zeta)}=0,\quad \{\{p,\gamma\},p\}|_{(\rho,\zeta)}>0,
\end{equation}
where $\{.,.\}$ denotes the Poisson braket. We say that a point $(\rho,\zeta)$ is hyperbolic if $x=0$ and $\tau^{2}>\eta^{2}$, so that there are two distinct nonzero real solutions $\xi$ to $\xi^{2}+(1+x)\eta^{2}-\tau^{2}=0$. 

Consider an approximate solution for \eqref{ondes} of the form
\begin{equation}\label{integral1}
\int e^{\frac{i}{h}(y\eta+t\tau+(x+1-\frac{\tau^{2}}{\eta^{2}})\xi+\frac{\xi^{3}}{3\eta^{2}})}g(t,\xi/\eta,\tau,h)\Psi(\eta)/\eta d\xi d\eta d\tau
\end{equation}
where the symbol $g$ is a smooth function independent of $x$, $y$ and where $\Psi\in C^{\infty}_{0}(\mathbb{R}^{*})$ is supported for $\eta$ in a small neighborhood of $1$, $0\leq\Psi(\eta)\leq 1$, $\Psi(\eta)=1$ for $\eta$ near $1$. This choice is motivated by the following: if $v(t,x,y)$ satisfies $(\partial^{2}_{t}-\partial^{2}_{x}-(1+x)\partial^{2}_{y})v=0$, then taking the Fourier transform in time $t$ and space $y$ we get $\partial^{2}_{x}\hat{v}=((1+x)\eta^{2}-\tau^{2})\hat{v}$, thus $\hat{v}$ can be expressed using Airy's function (given in Section \ref{secairy}) and its derivative. After the change of variables $\xi=\eta s$, the Lagrangian manifold associated to the phase function $\Phi$ of \eqref{integral1} will be given by
\begin{equation}
\Lambda_{\Phi}=\{(t,x,y,\tau=\partial_{t}\Phi,\xi=\partial_{x}\Phi=\eta s,\eta=\partial_{y}\Phi)|\partial_{s}\Phi=0,\partial_{\eta}\Phi=0\}\subset T^{*}\mathbb{R}^{3}\setminus 0.
\end{equation}
Let $\pi:\Lambda_{\Phi}\rightarrow\mathbb{R}^{3}$ be the natural projection and let $\Sigma$ denote the set of its singular points. %, i.e. the set of points where %$rang(d\pi)<3$.  Since $\Lambda_{\Phi}$ is parametrized by $(\tau,s,\eta)$, $\pi$ has Jacobian matrix
%\[
%\left(
%\begin{array}{ccc}
     % 1& 0 & 0   \\
     % 0& -2s & 0 \\
     % 0& 0& 1     
%\end{array}
%\right)
%\]
%and determinant $-2s$. The points where
%$\pi$ fails to be a diffeomorphism. %i.e. where the determinant vanishes, lie over the caustic set of $\Lambda_{\Phi}$, 
The points where the Jacobian of $d\pi$ vanishes lie over the caustic set, thus the fold set is given by $\Sigma=\{s=0\}$ and the caustic is defined by $\pi(\Sigma)=\{x+(1-\frac{\tau^{2}}{\eta^{2}})=0\}$.
%\begin{figure}
%\begin{center}
%\includegraphics[width=7cm]{bichar-1}
%\caption{Bicharacteristics of the half space}
%\label{fig}
%\end{center}
%\end{figure}

If on the boundary we are localized away from the caustic set $\pi(\Sigma)$, $\Lambda_{\Phi|_{x=0}}$ is the graph of a pair of canonical transformations, the billiard ball maps $\delta^{\pm}$. 
Roughly speaking, the billiard ball maps $\delta^{\pm}:T^{*}(\mathbb{R}\times\partial\Omega)\rightarrow T^{*}(\mathbb{R}\times\partial\Omega)$, defined on the hyperbolic region, continuous up to the boundary, smooth in the interior, are defined at a point of $T^{*}(\mathbb{R}\times\partial\Omega)$ by taking the two rays that lie over this point, in the hypersurface $\text{Char}(p)$, and following the null bicharacteristic through these points until you pass over $\{x=0\}$ again, projecting such a point onto $T^{*}(\mathbb{R}\times\partial\Omega)$ (a gliding point being "a diffractive point viewed from the other side of the boundary", there is no bicharacteristic in $T^{*}(\mathbb{R}\times\partial\Omega)$ through it, but in any neighborhood of a gliding point there are hyperbolic points).

In our model case the analysis is simplified by the presence of a large \emph{commutative} group of symmetries, the translations in $(y,t)$, and the billiard ball maps have specific formulas
\begin{equation}\label{billball}
\delta^{\pm}(y,t,\eta,\tau)=\Big(y\pm4(\frac{\tau^{2}}{\eta^{2}}-1)^{1/2}\pm\frac{8}{3}(\frac{\tau^{2}}{\eta^{2}}-1)^{3/2},t\mp 4(\frac{\tau^{2}}{\eta^{2}}-1)^{1/2}\frac{\tau}{\eta},\eta,\tau\Big).
\end{equation}
Away from $\pi(\Sigma)$ these maps have no recurrent points, since under iteration
$t((\delta^{\pm})^{n})\rightarrow\pm\infty$ as $n \rightarrow \infty$. 
The composite relation with $n$ factors
\[
\Lambda_{\Phi|_{x=0}}\circ...\circ\Lambda_{\Phi|_{x=0}}
\]
has, always away from $\pi(\Sigma)$, $n+1$ components, obtained namely using the graphs of the iterates
$(\delta^{+})^{n}, (\delta^{+})^{n-2},..,(\delta^{-})^{n}$,
\begin{equation}\label{deltait}
(\delta^{\pm})^{n}(y,t,\eta,\tau)=\Big(y\pm4n(\frac{\tau^{2}}{\eta^{2}}-1)^{1/2}\pm\frac{8}{3}n(\frac{\tau^{2}}{\eta^{2}}-1)^{3/2},t\mp 4n(\frac{\tau^{2}}{\eta^{2}}-1)^{1/2}\frac{\tau}{\eta},\eta,\tau\Big).
\end{equation}
All these graphs, of the powers of $\delta^{\pm}$, are disjoint away from $\pi(\Sigma)$ and locally finite, in the sense that only a finite number of components meet any compact subset of $\{\frac{\tau^{2}}{\eta^{2}}-1>0\}$. 
Since $(\delta^{\pm})^{n}$ are all immersed canonical relations, it is necessary to find a parametrization of each to get at least microlocal representations of the associated Fourier integral operators. We see that
\[
y\eta+t\tau+\frac{4}{3}\eta(\frac{\tau^{2}}{\eta^{2}}-1)^{3/2}, 
\]
are parametrizations of $\Lambda_{\Phi|_{x=0}}$, thus the iterated Lagrangians $(\Lambda_{\Phi|_{x=0}})^{\circ n}$ are parametrized by
\[
y\eta+t\tau+\frac{4}{3}n\eta(\frac{\tau^{2}}{\eta^{2}}-1)^{3/2},
\]
and the corresponding phase functions associated to $(\Lambda_{\Phi})^{\circ n}$ will be given by
\[
\Phi_{n}=\Phi+\frac{4}{3}n\eta (\frac{\tau^{2}}{\eta^{2}}-1)^{3/2}.
\] 

Let us come back to the wave equation \eqref{ondes} and describe the approximate solution we want to chose. The domain $\Omega$ being strictly convex, at each point on the boundary there exists a bicharacteristic that intersects the boundary $\mathbb{R}\times\partial\Omega$ tangentially having exactly second order contact with the boundary and remaining in the complement of $\mathbb{R}\times\bar{\Omega}$. Here we deal with $\gamma(\rho)=x$ and \eqref{hyconv2} translates into $x=\xi=0$, $|\tau|=|\eta|>0$. Let 
\[
(\rho_{0},\zeta_{0})=(0,0,0,0,1,-1)\in T^{*}(\mathbb{R}\times\partial\Omega).
\]
We shall place ourself in the region $\mathcal{V}_{a}$ near $(\rho_{0},\zeta_{0})$,
\[
\mathcal{V}_{a}=\{(\rho,\zeta)|\xi^{2}+(1+x)\eta^{2}-\tau^{2}=0, x=0,\tau^{2}=(1+a)\eta^{2}\},
\]
where $a=h^{\delta}$, $0<\delta<2/3$ will be chosen later and $\eta$ belongs to a neighborhood of $1$. Notice that, in some sense, $a$ measures the "distance" to the gliding point $(\rho_{0},\zeta_{0})$. 

Let $u_{h}$ be defined by
\begin{equation}\label{integral}
u_{h}(t,x,y)=\int e^{\frac{i}{h}(y\eta-t(1+a)^{1/2}\eta+(x-a)\xi+\frac{\xi^{3}}{3\eta^{2}})}g(t,\xi/\eta,h)\Psi(\eta)/\eta d\xi d\eta, 
\end{equation}
where the symbol $g$ is a smooth function independent of $x$, $y$ and where $\Psi\in C^{\infty}_{0}(\mathbb{R}^{*})$ is supported for $\eta$ in a small neighborhood of $1$, $0\leq\Psi(\eta)\leq 1$, $\Psi(\eta)=1$ for $\eta$ near $1$. We consider the sum
\[
U_{h}(t,x,y)=\sum_{n=0}^{N}u^{n}_{h}(t,x,y),\quad u^{n}_{h}(t,x,y)=\int e^{\frac{i}{h}\Phi_{n}}g^{n}(t,s,\eta,h)\Psi(\eta)d\eta ds,
\]
where $\Phi_{n}=\Phi+\frac{4}{3}n\eta a^{3/2}$ are the phase functions defined above such that $\Lambda_{\Phi_{n}}=(\Lambda_{\Phi})^{\circ n}$ and where the symbols $g^{n}$ will be chosen such that on the boundary the Dirichlet condition to be satisfied. At $x=0$ the phases have two critical, non-degenerate points, thus each $u^{n}_{h}$ writes as a sum of two trace operators, $Tr_{\pm}(u^{n}_{h})$, localized respectively for $y-(1+a)^{1/2}t+\frac{4}{3}na^{3/2}$ near $\pm\frac{2}{3}na^{3/2}$, and in order to obtain a contribution $O_{L^{2}}(h^{\infty})$ on the boundary we define the symbols such that $Tr_{-}(g^{n})+Tr_{+}(g^{n+1})=O_{L^{2}}(h^{\infty})$. This will be possible by Egorov theorem, as long as $N\ll a^{3/2}/h$. This last condition, together with the assumption of finite time (which implies $0<N(\frac{\tau^{2}}{\eta^{2}}-1)^{1/2}<\infty$) allows to estimate the number of reflections $N$.

The motivation of this construction comes from the fact that near the caustic set $\pi(\Sigma)$ one notices a singularity of cusp type for which one can estimate the $L^{r}(\Omega)$ norms. Moreover, if  at $t=0$ one considers symbols localized in a small neighborhood of the caustic set, then one can show that the respective "pieces of cusps" propagate until they reach the boundary but short after that their contribution becomes $O_{L^{2}}(h^{\infty})$, since as $t$ increases, $s$ takes greater values too and thus one quickly quits a neighborhood of the Lagrangian $\Lambda_{\Phi}$ which contains the semi-classical wave front set $WF_{h}(u_{h})$ of $u_{h}$. This argument is valid for all $u^{n}_{h}$, thus the approximate solutions  $u^{n}_{h}$ will have almost disjoints supports and the $L^{q}([0,1],L^{r}(\Omega))$ norms of the sum $U_{h}$ will be computed as the sum of the norms of each $u^{n}_{h}$ on small intervals of time of size $a^{1/2}$. 

\begin{figure}
\begin{center}
\includegraphics[width=15cm]{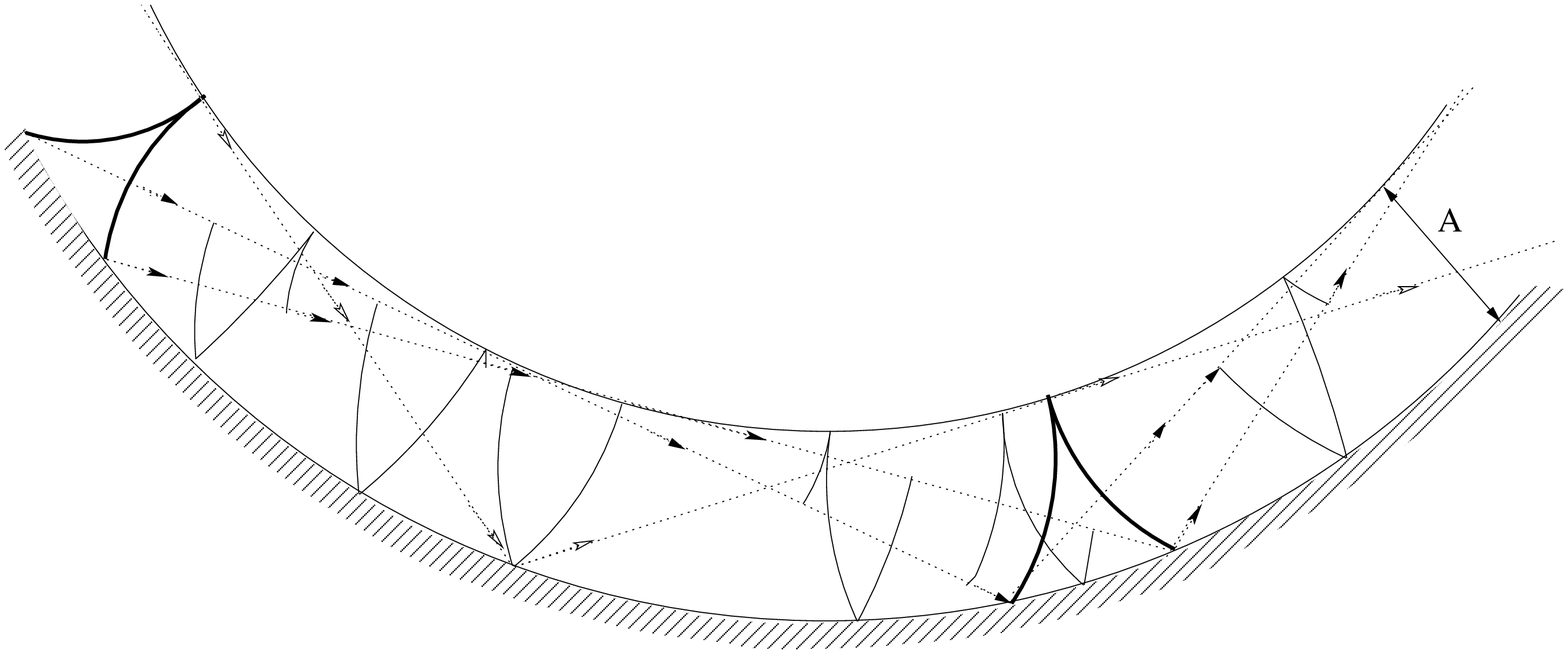}
\caption{Propagation of the cusp. \emph{A caustic is defined as the envelope of the rays which appear in a given problem: each ray is tangent to the caustic at a given point. If one assigns a direction on the caustic, it induces a direction on each ray. Each point outside the caustic lies on a ray which has left the caustic and also lies on a ray approaching the caustic. Each curve of constant phase has a cusp where it meets the caustic. }}
\label{fig}
\end{center}
\end{figure}

\subsection{Choice of the symbol}
Let $a=h^{\delta}$, $0<\delta<2/3$ to be chosen and let $u_{h}$ be given by the formula  \eqref{integral}. Applying the wave operator $\square$ to $u_{h}$ gives:
\begin{multline}\label{apros}
h^{2}\Box u_{h}=\int
e^{\frac{i}{h}\Phi}\Big(h^{2}\partial^{2}_{t}g-2ih\eta(1+a)^{1/2}\partial_{t}g+
\eta^{2}(x-a+s^{2})g\Big)\Psi(\eta)ds d\eta\\
=\int
e^{\frac{i\eta}{h}(y-t(1+a)^{1/2}+s(x-a)+\frac{s^{3}}{3})}\Big(h^{2}\partial^{2}_{t}g+ih\eta(\partial_{s}g
-2(1+a)^{1/2}\partial_{t}g)\Big)\eta\Psi(\eta)ds d\eta.
\end{multline}
\begin{dfn}\label{dfns}
Let $\lambda\geq 1$. For a given compact $K\subset\mathbb{R}$ we define the space $\mathcal{S}_{K}(\lambda)$, consisting of functions $\varrho(z,\lambda)\in C^{\infty}(\mathbb{R})$ which satisfy
\begin{enumerate}
\item  $\sup_{z\in\mathbb{R},\lambda\geq 1}|\partial^{\alpha}_{z}\varrho(z,\lambda)|\leq C_{\alpha}$, where $C_{\alpha}$ are constants independent of $\lambda$,
\item If $\psi(z)\in C^{\infty}_{0}$ is a smooth function equal to $1$ in a neighborhood of $K$, $0\leq\psi\leq 1$ then $(1-\psi)\varrho\in O_{\mathcal{S}(\mathbb{R})}(\lambda^{-\infty})$.
\end{enumerate}
An example of function $\varrho(z,\lambda)\in \mathcal{S}_{K}(\lambda)$, $K\subset\mathbb{R}$ is the following: let $k(z)$ be the smooth function on $\mathbb{R}$ defined by
\[
k(z)=\left\{
                \begin{array}{ll}
               c\exp{(-1/(1-|z|^{2}))},\quad\ if\ |z|<1,\\
               0,\quad \ if \ |z| \geq 1,
                \end{array}
                \right.
\]
where $c$ is a constant chosen such that $\int_{\mathbb{R}}k(z)dz=1$. Define a mollifier $k_{\lambda}(z):=\lambda k(\lambda z)$ and let $\tilde{\varrho}\in C^{\infty}_{0}(K)$ be a smooth function with compact support included in $K$. If we set $\varrho(z,\lambda)=(\tilde{\varrho}*k_{\lambda})(z)$, then one can easily check that $\varrho$ belongs to $\mathcal{S}_{K}(\lambda)$.
\end{dfn}
Let $\lambda=\lambda(h)=h^{3\delta/2-1}$, $K_{0}=[-c_{0},c_{0}]$ for some small $0<c_{0}<1$ and let $\varrho(.,\lambda)\in \mathcal{S}_{K_{0}}(\lambda)$ be a smooth function. We define
\begin{equation}\label{gforme}
g(t,s,h)=\varrho(\frac{t+2(1+a)^{1/2}s}{2(1+a)^{1/2}a^{1/2}},\lambda).
\end{equation}
Notice that $t+2(1+a)^{1/2}s$ is an integral curve of the vector field
 $\partial_{s}-2(1+a)^{1/2}\partial_{t}$, thus inserting \eqref{gforme} in \eqref{apros} gives
\begin{equation}\label{boxx}
\Box u_{h}=h^{-\delta}(4(1+a))^{-1}\int e^{\frac{i}{h}\Phi}\partial^{2}_{1}\varrho(\frac{t+2(1+a)^{1/2}s}{2(1+a)^{1/2}a^{1/2}},\lambda)\Psi(\eta)dsd\eta.
\end{equation}

\subsection{The boundary condition}\label{bdcond}
We compute $u_{h}$ on the boundary. We make the change of variables $s=a^{1/2}v$ in the integral defining $u_{h}(t,0,y)$ and set $z=\frac{t}{2(1+a)^{1/2}a^{1/2}}$. Then
\begin{multline}
u_{h}(t,0,y)
=a^{1/2}\int_{\eta}\frac{\eta\lambda}{2\pi}e^{\frac{i\eta}{h}(y-t(1+a)^{1/2})}\times\\
\times\Big(\int_{\zeta}\int_{v}
e^{i\eta\lambda(\frac{v^{3}}{3}-v(1-\zeta))}dv\int_{z'} e^{i\eta\lambda(z-z')\zeta}\varrho(z',\lambda)dz'd\zeta\Big)\Psi(\eta)d\eta\\
=a^{1/2}\int_{\eta} (\eta\lambda)^{2/3}e^{\frac{i\eta}{h}(y-t(1+a)^{1/2})}\Psi(\eta)\times\\
\times \int_{\zeta,z'}e^{i\eta\lambda(z-z')\zeta}Ai(-(\eta\lambda)^{2/3}(1-\zeta))\varrho(z',\lambda)dz'd\zeta d\eta.
\end{multline}
For $\eta\in \text{supp}(\Psi)$ we introduce
\begin{equation}\label{opin}
I(\varrho(.,\lambda))_{\eta}(z,\lambda):=\frac{(\eta\lambda)^{7/6}}{2\pi}\int_{\zeta,z'}e^{i\eta\lambda(z-z')\zeta}Ai(-(\eta\lambda)^{2/3}(1-\zeta))\varrho(z',\lambda)dz'd\zeta,
\end{equation}
then 
\begin{equation}\label{fopin}
\Psi(\eta)(I(\varrho(.,\lambda))_{\eta})^{\wedge}(\eta\lambda\zeta,\lambda)=(\eta\lambda)^{1/6}\Psi(\eta)Ai(-(\eta\lambda)^{2/3}(1-\zeta))\hat{\varrho}(\eta\lambda\zeta,\lambda).
\end{equation}
The next Lemma shows that the symbol of the operator defined in \eqref{opin} is localized for $\zeta$ as close as we want to $0$.
\begin{lemma}\label{lem3}
For $\varrho(.,\lambda)\in \mathcal{S}_{K}(\lambda)$ for some compact $K$ then $(I(\varrho)_{\eta})^{\wedge}(.,\lambda)$ defined by \eqref{opin}, \eqref{fopin} is localized near $\zeta=0$, more precisely, if $\chi$ is a smooth function with support included in a small neighborhood  $(-2c,2c)$ of $0$, $0<c\leq 1/4$, $\chi|_{[-c,c]}=1$, $0\leq\chi\leq 1$, then we have for $\eta\in \text{supp}(\Psi)$
\begin{multline}
I(\varrho(.,\lambda))_{\eta}(z,\lambda)
=\frac{(\eta\lambda)^{7/6}}{2\pi}\int_{\zeta,z'}e^{i\eta\lambda(z-z')\zeta}Ai(-(\eta\lambda)^{2/3}(1-\zeta))\times\\
\times\chi(\zeta)\varrho(z',\lambda)dz'd\zeta+O_{\mathcal{S}(\mathbb{R})}((\eta\lambda)^{-\infty}).
\end{multline}
\end{lemma}
\begin{proof}
Let $\varrho(.,\lambda)\in \mathcal{S}_{K}(\lambda)$. If we set, for $\eta\in \text{supp}(\Psi)$
\[
J(\varrho(.,\lambda))_{\eta}(z,\lambda):=\frac{(\eta\lambda)^{7/6}}{2\pi}\int_{\zeta,z'}e^{i\eta\lambda(z-z')\zeta}Ai(-(\eta\lambda)^{2/3}(1-\zeta))(1-\chi(\zeta))\varrho(z',\lambda)dz'd\zeta
\]
we need to prove the following
\[
\Psi(\eta)J(\varrho(.,\lambda))_{\eta}(z,\lambda)\in \Psi(\eta) O_{\mathcal{S}(\mathbb{R}_{z})}((\eta\lambda)^{-\infty})=O_{\mathcal{S}(\mathbb{R}_{z})}(\lambda^{-\infty}),
\]
which is the same as to show that $(J(\varrho))^{\wedge}_{\eta}(\xi,\lambda)\in O_{\mathcal{S}(\mathbb{R}_{\xi})}(\lambda^{-\infty})$ or equivalently that 
\begin{equation}\label{compexplitop}
\Psi(\eta)J(\varrho)^{\wedge}_{\eta}(\eta\lambda\zeta,\lambda)\in \Psi(\eta) O_{\mathcal{S}(\mathbb{R}_{\zeta})}((\eta\lambda)^{-\infty}).
\end{equation}
In order to prove \eqref{compexplitop} we first compute $(J(\varrho))^{\wedge}_{\eta}(\eta\lambda\zeta,\lambda)$ explicitly:
\begin{multline}\label{psijrofour}
\Psi(\eta)(J(\varrho(.,\lambda)))^{\wedge}_{\eta}(\eta\lambda\zeta,\lambda)
=\\=(\eta\lambda)^{1/6}\Psi(\eta)
 Ai(-(\eta\lambda)^{2/3}(1-\zeta))(1-\chi(\zeta))\hat{\varrho}(\eta\lambda\zeta,\lambda)
\end{multline}
It remains to show that the right hand side of \eqref{psijrofour} belongs to $\Psi(\eta)O_{\mathcal{S}(\mathbb{R}_{z})}((\eta\lambda)^{-\infty})$. Notice that this will conclude the proof of the Lemma \ref{lem3}. If $\chi(\zeta)\neq 1$ then $\zeta$ lies outside a neighborhood of $0$, $|\zeta|\geq c$ and for $\eta\in \text{supp} (\Psi)$ we can perform integrations by parts in the integral defining $\hat{\varrho}(\eta\lambda\zeta,\lambda)$:
\begin{multline}\label{four}
\Psi(\eta)\hat{\varrho}(\eta\lambda\zeta,\lambda)=\Psi(\eta)\int_{z'}e^{-i\eta\lambda\zeta z'}\varrho(z',\lambda)dz'=\\=\frac{\Psi(\eta)}{(i\eta\lambda\zeta)^{m}}\int_{z'}e^{-i\eta\lambda\zeta z'}\partial^{m}_{z'}\varrho(z',\lambda)dz'.
\end{multline}
Writing $\varrho(z',\lambda)=\psi(z')\varrho(z',\lambda)+(1-\psi(z'))\varrho(z',\lambda)$ for some smooth cutoff function $\psi$ equal to $1$ on $K$ and using that  $\|\partial^{m}_{z'}(\psi\varrho)(.,\lambda)\|_{L^{\infty}(\mathbb{R})}\leq C'_{m}$ for some constants $C'_{m}$ independent of $\lambda$ and that, on the other hand $\partial^{m}_{z'}((1-\psi)\varrho(.,\lambda))\in O_{\mathcal{S}(\mathbb{R})}(\lambda^{-\infty})$, we deduce the desired result.
\end{proof}

In what follows we use the results in Section \ref{secairy} of the Appendix in order to write, for $\zeta$ close to $0$
\[
Ai(-(\eta\lambda)^{2/3}(1-\zeta))=A^{+}(-(\eta\lambda)^{2/3}(1-\zeta))+A^{-}(-(\eta\lambda)^{2/3}(1-\zeta))
\]
where $A^{\pm}$ have the following asymptotic expansions
\begin{multline}
A^{\pm}(-(\eta\lambda)^{2/3}(1-\zeta))\simeq (\eta\lambda)^{-1/6}(1-\zeta)^{-1/4}\times\\ \times e^{\mp\frac{2i}{3}\eta\lambda(1-\zeta)^{3/2}\pm\frac{i\pi}{2}-\frac{i\pi}{4}}(\sum_{j\geq 0}\frac{a_{\pm,j}(-1)^{-j/2}(1-\zeta)^{-3j/2}}{(\eta\lambda)^{j}}).
\end{multline}
We obtain two contributions in $I(\varrho(.,\lambda))_{\eta}(.,\lambda)$ which we denoted
\begin{equation}
\frac{(\eta\lambda)^{7/6}}{2\pi}\int_{\zeta,z'}e^{i\eta\lambda(z-z')\zeta}A^{\pm}(-(\eta\lambda)^{2/3}(1-\zeta))\chi(\zeta)\varrho(z',\lambda)dz'd\zeta.
\end{equation}
We can summarize the preceding results as follows:
\begin{prop}\label{bound}
On the boundary $u_{h}|_{x=0}$ writes (modulo $O_{\mathcal{S}(\mathbb{R})}(\lambda^{-\infty})$) as a sum of two trace operators,
\begin{equation}\label{boun}
u_{h}(t,0,y)=Tr_{+}(u_{h})(t,y;h)+Tr_{-}(u_{h})(t,y;h),
\end{equation}
where, for $z=\frac{t}{2(1+a)^{1/2}a^{1/2}}$,
\begin{equation}
Tr_{\pm}(u_{h})(t,y;h)=2\pi\sqrt{\frac{a}{\lambda}}\int
e^{\frac{i\eta}{h}(y-t(1+a)^{1/2}\mp\frac{2}{3}a^{3/2})}\eta^{-1/2}\Psi(\eta)I_{\pm}(\varrho)_{\eta}(z,\lambda)d\eta,
\end{equation}
\begin{multline}
I_{\pm}(\varrho(.,\lambda))_{\eta}(z,\lambda)=e^{\pm i\pi/2-i\pi/4}\frac{\eta\lambda}{2\pi}\int_{\zeta,z'}e^{i\eta\lambda(\zeta(z-z')\mp\frac{2}{3}((1-\zeta)^{3/2}-1))}\times\\ \times \chi(\zeta)a_{\pm}(\zeta,\eta\lambda)\varrho(z',\lambda)dz'd\zeta,
\end{multline}
and where $a_{\pm}$ are the symbols of the Airy functions $A^{\pm}$ 
\begin{equation}\label{simsim}
a_{\pm}(\zeta,\eta\lambda)\simeq(1-\zeta)^{-1/4}(\sum_{j\geq 0}\frac{a_{\pm,j}(-1)^{-j/2}(1-\zeta)^{-3j/2}}{(\eta\lambda)^{j}}),
\end{equation}
where $a_{\pm,j}$ are given in \eqref{simai}.
\end{prop}
We also need the next Lemma:
\begin{lemma}\label{lemlocal}
Let $p\in\mathbb{Z}$ and for some $0<c_{0}<1$ set $K_{p}=[-c_{0}+p,c_{0}+p]$. Then for $\eta$ belonging to the support of $\Psi$ we have $I_{\pm,\eta}:\mathcal{S}_{K_{p}}(\lambda)\rightarrow\mathcal{S}_{K_{p\mp1}}(\lambda)$.
\end{lemma}
\begin{proof}
The phase functions in $I_{\pm}(\varrho(.,\lambda))_{\eta}(z,\lambda)$ are given by 
\[
\eta\phi_{\pm}(z,z',\zeta)=\eta((z-z')\zeta\mp\frac{2}{3}((1-\zeta)^{3/2}-1)),
\]
with critical points satisfying
\[
\partial_{\zeta}\phi_{\pm}(z,z',\zeta)=z-z'\pm(1-\zeta)^{1/2}=0,\quad \partial_{z'}\phi_{\pm}(z,z',\zeta)=-\zeta=0.
\]
Outside small neighborhoods of $\zeta=0$ and $z'=z\pm(1-\zeta)^{1/2}$ we make integrations by parts in order to obtain a small contribution $\Psi(\eta)O_{\mathcal{S}(\mathbb{R})}((\eta\lambda)^{-\infty})$. Indeed, if we write
\begin{equation}
\Psi(\eta)I_{\pm}(\varrho(.,\lambda))_{\eta}(z,\lambda)=
\end{equation}
\[
=e^{\pm i\pi/2-i\pi/4}\Psi(\eta)\frac{\eta\lambda}{2\pi}\int_{\zeta,z'}e^{i\eta\lambda((z-z')\zeta\mp\frac{2}{3}((1-\zeta)^{3/2}-1))}a_{\pm}(\zeta,\eta\lambda)\chi(\zeta)\varrho(z',\lambda)dz'd\zeta,
\]
where $a_{\pm}$ are given in \eqref{simsim}, we have to check that the conditions of Definition \ref{dfns} are satisfied for $I_{\pm}(\varrho(.,\lambda))_{\eta}(z,\lambda)$ and $\mathcal{S}_{K_{p\mp1}}(\lambda)$:
\begin{itemize}
\item First we prove that for $\eta\in \text{supp}(\Psi)$
\begin{equation}\label{cond1}
\sup_{z\in\mathbb{R},\lambda\geq 1}|\partial^{\alpha}_{z}I_{\pm}(\varrho)_{\eta}(z,\lambda)|\leq C^{1}_{\alpha}.
\end{equation}
For $\eta\in \text{supp}(\Psi)$ we have
\begin{multline}
\Psi(\eta)\partial^{\alpha}_{z}I_{\pm}(\varrho)_{\eta}(z,\lambda)
=e^{\pm i\pi/2-i\pi/4}\Psi(\eta)\frac{\eta\lambda}{2\pi}\int_{\zeta,z'}e^{i\eta\lambda((z-z')\zeta\mp\frac{2}{3}((1-\zeta)^{3/2}-1))}\times\\ \times (i\eta\lambda\zeta)^{\alpha}a_{\pm}(\zeta,\lambda)\chi(\zeta)\varrho(z',\lambda)dz'd\zeta,
\end{multline}
and we shall split the integral in $\zeta$ in two parts, according to $\lambda\zeta\leq2$ or $\lambda\zeta>2$ for $\eta$ on the support of $\Psi$: in the first case there is nothing to do, the change of variables $\xi=\eta\lambda\zeta$ allowing to obtain bounds of type  \eqref{cond1}. In case $\lambda\zeta>2$ we make integrations by parts in the integral defining $\hat{\varrho}(\eta\lambda\zeta,\lambda)$ like in \eqref{four} in order to conclude.

\item Secondly, let $\psi_{\pm}$ be smooth cutoff functions equal to $1$ in small neighborhoods of $K_{p\mp 1}$ and such that $0\leq \psi_{\pm}\leq 1$. We prove that
\[
(1-\psi_{\pm}(z))\Psi(\eta)I_{\pm}(\varrho)_{\eta}(z,\lambda)=\Psi(\eta)O_{\mathcal{S}(\mathbb{R})}((\eta\lambda)^{-\infty}).
\]
Since $\psi_{\pm}$ equal to $1$ on some neighborhoods of  $K_{p\mp 1}$ there exist $c'>0$ small enough such that $\psi_{\pm}|_{[-c_{0}+p\mp1-5c',c_{0}+p\mp1+5c']}=1$.
Since $(I(\varrho)_{\eta})^{\wedge}$ is localized as close as we want to $\zeta=0$ then from the proof of Lemma \ref{lem3} we can find some (other) smooth function $\tilde{\chi}$ with support included in $(-2c',2c')$, equal to $1$ on $[-c',c']$ so that
\[
\Psi(\eta)(I(\varrho)_{\eta})^{\wedge}(\eta\lambda\zeta,\lambda)=\Psi(\eta)(I(\varrho)_{\eta})^{\wedge}(\eta\lambda\zeta,\lambda)\tilde{\chi}(\zeta)+\Psi(\eta)O_{\mathcal{S}(\mathbb{R})}((\eta\lambda)^{-\infty}).
\]
Let $\psi\in C^{\infty}_{0}$ with support included in $(-c_{0}+p-c',c_{0}+p+c')$. We split $\varrho=\psi\varrho+(1-\psi)\varrho$ and since $(1-\psi)\varrho(.,\lambda)$ belongs to $O_{\mathcal{S}(\mathbb{R})}(\lambda^{-\infty})$ it is enough to prove the preceding assertion with $\varrho$ replaced by $\psi\varrho$. On the support of $\psi\varrho$ we have
$|z'-p|\leq c_{0}+c'$
and on the support of $1-\psi_{\pm}$ we have $|z-p\pm1|>c_{0}+5c'$. On the other hand, if $c'$ is chosen small enough then on the support of $\psi$ we have $1-3c'\leq(1-\zeta)^{1/2}\leq 1+3c'$, thus we can make integrations by parts in the integral in $\zeta$ since in the region we consider we have $|\partial_{\zeta}\phi_{\pm}|\geq p\mp1+c_{0}+c'-p-c_{0}-c'\pm 1-3c'\geq c'$. From the discussion above and
\begin{multline}\label{init}
(1-\psi_{\pm}(z))I_{\pm}(\varrho)_{\eta}(z,\lambda)=
e^{\pm i\pi/2-i\pi/4}(1-\psi_{\pm}(z))\times \\
\times \Psi(\eta)\frac{\eta\lambda}{2\pi}\int_{\zeta,z'}e^{i\eta\lambda((z-z')\zeta\mp\frac{2}{3}((1-\zeta)^{3/2}-1))}a_{\pm}(\zeta,\eta\lambda)\tilde{\chi}(\zeta)\psi(z')\varrho(z',\lambda)dz'd\zeta
\end{multline}
we conclude by performing integrations by parts in $\zeta$. In fact, we could have noticed from the beginning that, inserting under the integral \eqref{init} a cut-off localized close to $\zeta=0$, $z'=z$ and performing integrations by parts, one makes appear a factor bounded by $(1+\lambda|\eta||\zeta|)^{-N}$ for all $N\geq 0$. 
\end{itemize}
\end{proof}

\subsubsection{Construction of the approximate solution}
Let $p\in\mathbb{Z}$ and $K_{p}=[-c_{0}+p,c_{0}+p]$. For $\eta\in \text{supp}(\Psi)$, some $\tilde{\lambda}\geq 1$ and $\varrho(.,\tilde{\lambda})\in\mathcal{S}_{K_{0}}(\tilde{\lambda})$ write 
\begin{equation}\label{opaaa}
I_{\pm}(\varrho(.,\tilde{\lambda}))_{\eta}(z,\lambda)=e^{\pm i\pi/2-i\pi/4}\frac{\eta\lambda}{2\pi}\int e^{i\eta\lambda(z\zeta-\psi_{\pm}(z',\zeta))}\chi(\zeta)a_{\pm}(\zeta,\eta\lambda)\varrho(z',\tilde{\lambda})dz'd\zeta,
\end{equation}
where we set
\begin{equation}\label{symba}
\psi_{\pm}(z',\zeta)=z'\zeta\pm\frac{2}{3}((1-\zeta)^{3/2}-1).
\end{equation}
We want to apply the Egorov theorem in order to invert the operators $I_{\pm,\eta}$.
The symbols $\chi(\zeta)a_{\pm}(\zeta,\eta\lambda)$ are elliptic at $\zeta=0$, consequently (eventually shrinking the support of $\chi$) there exists symbols $b_{\pm}(\zeta,\eta\lambda)$ which are asymptotic expansions in $(\eta\lambda)^{-1}$ for $\eta$ belonging to the support of $\Psi$,
%\[
%b(\zeta,\eta\lambda)=\sum_{j=0}^{\infty}b_{j}(\zeta)(\eta\lambda)^{-j}, \quad b_{0}(\zeta)=\frac{\chi(\zeta)}{a_{0}},
%\] 
such that, if one denotes by $J_{\pm}(.)_{\eta}$ the operators defined for $\breve{\varrho}\in\mathcal{S}_{K\mp 1}(\tilde{\lambda})$ by 
\begin{equation}\label{opbbb}
J_{\pm}(\breve{\varrho}(.,\tilde{\lambda}))_{\eta}(z',\lambda)=e^{\mp i\pi/2+i\pi/4}\frac{\eta\lambda}{2\pi}\int e^{i\eta\lambda(\psi_{\pm}(z',\zeta)-z\zeta)}b_{\pm}(\zeta,\eta\lambda)\breve{\varrho}(z,\tilde{\lambda})dzd\zeta,
\end{equation}
then one has 
\[
\breve{\varrho}(.,\tilde{\lambda})=I_{+}(J_{+}(\breve{\varrho}(.,\tilde{\lambda}))_{\eta}(.,\lambda))_{\eta}(.,\lambda)+O_{\mathcal{S}(\mathbb{R})}(\lambda^{-\infty})+O_{\mathcal{S}(\mathbb{R})}(\tilde{\lambda}^{-\infty}),
\]
\[
\varrho(.,\tilde{\lambda})=J_{+}(I_{+}(\varrho(.,\tilde{\lambda}))_{\eta}(.,\lambda))_{\eta}(.,\lambda)+O_{\mathcal{S}(\mathbb{R})}(\lambda^{-\infty})+O_{\mathcal{S}(\mathbb{R})}(\tilde{\lambda}^{-\infty}),
\]
 and also 
\[ 
\breve{\varrho}(.,\tilde{\lambda})=I_{-}(J_{-}(\breve{\varrho}(.,\tilde{\lambda}))_{\eta}(.,\lambda))_{\eta}(.,\lambda)+O_{\mathcal{S}(\mathbb{R})}(\lambda^{-\infty})+O_{\mathcal{S}(\mathbb{R})}(\tilde{\lambda}^{-\infty}), 
\]
\[
\varrho(.,\tilde{\lambda})=J_{-}(I_{-}(\varrho(.,\tilde{\lambda}))_{\eta}(.,\lambda))_{\eta}(.,\lambda)+O_{\mathcal{S}(\mathbb{R})}(\lambda^{-\infty})+O_{\mathcal{S}(\mathbb{R})}(\tilde{\lambda}^{-\infty}).
\] 
\begin{rmq}\label{rmqc}
A direct computation shows that, for instance
\[
J_{\pm}(I_{\pm}(\varrho(.,\tilde{\lambda}))_{\eta}(.,\lambda))_{\eta}(z,\lambda)=\frac{\eta\lambda}{2\pi}\int e^{i\eta\lambda(z-z')\zeta}\chi(\zeta)a_{\pm}(\zeta,\eta\lambda)b_{\pm}(\zeta,\eta\lambda)\varrho(z',\tilde{\lambda})dz' d\zeta
\]
and consequently (since the coefficients do not depend on $z'$ and because of the expression of the phase functions $\psi_{\pm}(z',\zeta)$) one can take $b_{\pm}(\zeta,\eta\lambda)=\frac{\chi(\zeta)}{a_{\pm}(\zeta,\eta\lambda)}$.
\end{rmq}
\begin{prop}\label{propimportant}
Let $N\simeq \lambda h^{\epsilon}$ for some small $\epsilon>0$ and $1\leq n\leq N$. Let $T_{k}$ be the translation operator which to a given $\varrho(z)$ associates $\varrho(z+k)$. Then for $\eta\in \text{supp}(\Psi)$
\begin{equation}\label{t1jit1n}
(T_{1}\circ J_{+}(.)_{\eta}\circ I_{-}(.)_{\eta}\circ T_{1})^{n}:\mathcal{S}_{K_{0}}(\lambda)\rightarrow\mathcal{S}_{K_{0}}(\lambda/n)\quad \text{uniformly in} \quad n.
\end{equation}
Notice that since $\lambda/n \geq h^{-\epsilon}\gg 1$, then one has
\begin{equation}\label{equivohinf}
O_{\mathcal{S}(\mathbb{R})}(\lambda^{-\infty})=O_{\mathcal{S}(\mathbb{R})}((\lambda/n)^{-\infty})=O_{\mathcal{S}(\mathbb{R})}(h^{\infty}).
\end{equation}
\end{prop}
\begin{rmq}
Notice that at this point we have a restriction on the number of reflections $N$ which should be much smaller when compared to $\lambda=a^{3/2}/h$. In fact, in the proof of Proposition \ref{propimportant} we apply the stationary phase with parameter $\lambda/n$ which should be large enough, more precisely it should be larger than some (positive) power of $h^{-1}$. Using \eqref{equivohinf}, this would imply that away from the critical points the contributions of the oscillatory integrals are $O(h^{\infty})$.
\end{rmq}
\begin{proof}
We start by determining the explicit form of the operator in \eqref{t1jit1n}.

For $\zeta$ on the support of $\chi$, $|\zeta|\leq 2c<1$ for $c$ small enough, set
\begin{equation}
f(\zeta)=2\zeta-\frac{4}{3}(1-(1-\zeta)^{\frac{3}{2}})=\frac{\zeta^{2}}{2}+O(\zeta^{3}).
\end{equation}
Let $\varrho(.,\lambda)\in\mathcal{S}_{K_{0}}(\lambda)$. Then for $\eta$ on the support of $\Psi$ we have
\begin{multline}
T_{1}(J_{+}(I_{-}(T_{1}(\varrho(.,\lambda)))_{\eta})_{\eta})(z)=\frac{\eta\lambda}{2\pi}\int_{z',\zeta} e^{i\eta\lambda((z-z')\zeta+f(\zeta))}c(\zeta,\eta\lambda)\varrho(z',\lambda)dz'd\zeta\\=(F_{\eta\lambda}*\varrho(.,\lambda))(z),
\end{multline}
where we set, for $\eta$ on the support of $\Psi$
\begin{equation}
F_{\eta\lambda}(z)=\frac{\eta\lambda}{2\pi}\int_{\zeta} e^{i\eta\lambda (z\zeta+f(\zeta))}c(\zeta,\eta\lambda)d\zeta,
\end{equation}
 \begin{equation}\label{symbc}
 c(\zeta,\eta\lambda)=\chi(\zeta)a_{+}(\zeta,\eta\lambda)b_{-}(\zeta,\eta\lambda)=\chi^{2}(\zeta)\sum_{j\geq 0}c_{j}(1-\zeta)^{-3j/2}(\eta\lambda)^{-j},\quad c_{0}=1,
 \end{equation}
thus for each $n$ we can write 
\begin{equation}\label{convo}
(T_{1}\circ J_{+}(.)_{\eta}\circ I_{-}(.)_{\eta}\circ T_{1})^{n}(\varrho(.,\lambda))(z)=(F_{\eta\lambda})^{*n}*\varrho(.,\lambda)(z).
\end{equation}
We can explicitly compute $(F_{\eta\lambda})^{*n}$
\begin{equation}\label{gmare}
(F_{\eta\lambda})^{* n}(z)=\frac{\eta\lambda}{2\pi}\int_{\zeta} e^{i\eta\lambda z\zeta}\hat{F}_{\eta\lambda}^{n}(\eta\lambda\zeta)d\zeta=
\frac{\eta\lambda}{2\pi}\int_{\zeta} e^{i\eta\lambda(z\zeta+nf(\zeta))}c^{n}(\zeta,\eta\lambda)d\zeta. 
\end{equation}
Set $\tilde{\zeta}=n\zeta$ and $\tilde{\lambda}=\frac{\lambda}{n}$. The choice we made for $N$ allows to write the right hand side of \eqref{gmare} as
\begin{equation}\label{fn}
(F_{\eta\lambda})^{*n}(z)=\frac{\eta\tilde{\lambda}}{2\pi}\int_{\tilde{\zeta}} e^{i\eta\tilde{\lambda}(z\tilde{\zeta}+n^{2} f(\frac{\tilde{\zeta}}{n}))}c^{n}(\frac{\tilde{\zeta}}{n},n\eta\tilde{\lambda})d\tilde{\zeta},
\end{equation}
therefor we obtain
\begin{multline}\label{varnbr}
(T_{1}\circ J_{+}(.)_{\eta}\circ I_{-}(.)_{\eta}\circ T_{1})^{n}(\varrho(.,\lambda))(z)=\\=\frac{\eta\tilde{\lambda}}{2\pi}\int_{\tilde{\zeta}} e^{i\eta\tilde{\lambda}((z-z')\tilde{\zeta}+n^{2} f(\frac{\tilde{\zeta}}{n}))}c^{n}(\frac{\tilde{\zeta}}{n},n\eta\tilde{\lambda})\varrho(z',\lambda)d\tilde{\zeta}dz'.
\end{multline}
The phase function in \eqref{varnbr} is given by
\[
\eta\phi_{n}(z,z',\tilde{\zeta})=\eta((z-z')\tilde{\zeta}+n^{2} f(\frac{\tilde{\zeta}}{n})),
\]
and its critical points satisfy
\begin{equation}
\partial_{\tilde{\zeta}}\phi_{n}=z-z'+nf'(\frac{\tilde{\zeta}}{n})=0,\quad \partial_{z'}\phi_{n}=-\tilde{\zeta}=0.
\end{equation}
\begin{itemize}
\item In order to show that for all $\alpha\geq 0$ there exists constants $C^{2}_{\alpha}$ independent of $n$, $\lambda$, such that
\[
\sup_{z\in\mathbb{R},\tilde{\lambda}=\lambda/n\geq 1}|\partial^{\alpha}_{z}(T_{1}\circ J_{+}(.)_{\eta}\circ I_{-}(.)_{\eta}\circ T_{1})^{n}(\varrho(.,\lambda))(z)|\leq C^{2}_{\alpha}
\]
we write
\[
\partial^{\alpha}_{z}(T_{1}\circ J_{+}(.)_{\eta}\circ I_{-}(.)_{\eta}\circ T_{1})^{n}(\varrho(.,\lambda))(z)=(F_{\eta\lambda})^{*n}*\partial^{\alpha}_{z}\varrho(.,\lambda)(z).
\]
For $\tilde{\zeta}$ outside a small neighborhood of $0$, $|\tilde{\zeta}|\geq c$, we perform integrations by parts in $z'$ in the integral \eqref{fn} defining $(F_{\eta\lambda})^{*n}(z)$ and obtain a contribution arbitrarily small. For $|\tilde{\zeta}|<2c$ small, let $\psi$ be a smooth function with support included in a $c$-neighborhood of $K_{0}$ and such that $(1-\psi)\varrho(.,\lambda)=O_{\mathcal{S}(\mathbb{R})}(\lambda^{-\infty})$. For $z$ away from a $5 c$-neighborhood of $K_{0}$ we saw in the proof of Lemma \ref{lemlocal} that we have $|\partial_{\tilde{\zeta}}\phi_{n}(z,z',\tilde{\zeta})|\geq c$ and we conclude again by integrations by parts in $\tilde{\zeta}$. Near the critical points $\tilde{\zeta}=0$ and $z=z'-nf'(\frac{\tilde{\zeta}}{n})$ we can apply the stationary method lemma in both variables $z'$, $\tilde{\zeta}$, uniformly in $n$: it is crucial here that $f(0)=f'(0)=0$, $f''(0)=1$ which gives, setting $g_{n}(\tilde{\zeta})=n^{2}f(\frac{\tilde{\zeta}}{n})$, 
\[
|g_{n}(\tilde{\zeta})|\leq d_{0}|\tilde{\zeta}|^{2},\quad |g'_{n}(\tilde{\zeta})|\leq d_{1}|\tilde{\zeta}|,\quad |g^{(k)}_{n}(\tilde{\zeta})|\leq d_{k},\forall k\geq 2,
\]
with constants $d_{k}$ independent of $n$ (we deal with Fourier multipliers), $|g''_{n}|\geq d'_{2}>0$ and that, on the other hand, from \eqref{symbc} we have $|\partial^{m}_{\tilde{\zeta}}\chi^{2n}(\frac{\tilde{\zeta}}{n})|\leq e_{m}$ for all $m\geq 0$ with constants $e_{m}$ independent of $n$ and 
\[
|\partial_{\tilde{\zeta}}c^{n}(\frac{\tilde{\zeta}}{n},n\eta\tilde{\lambda})|\leq e_{1}|c^{n}(\frac{\tilde{\zeta}}{n},n\eta\tilde{\lambda}|+e_{0}|\sum_{j\geq 0} c_{j}\frac{3j}{2}\frac{(1-\frac{\tilde{\zeta}}{n})^{-(3j+2)/2}}{(n\eta\tilde{\lambda})^{j}}|.
\]
%we split the integral in $\tilde{\zeta}$ in two parts, according to $\tilde{\lambda}\tilde{\zeta}\leq 2$ or $\tilde{\lambda}\tilde{\zeta}>2$. In the first case the integral is obviously bounded with constants independent of $n$, $\lambda$; in the second case we perform integrations by parts like in \eqref{four} (notice that here we used the fact that $c^{n}=\chi^{n}(1+O(\tilde{\lambda}^{-1}))$ and $0\leq \chi\leq 1$).

\item To check that for a smooth function $\tilde{\psi}$ equal to $1$ in a neighborhood of $K_{0}$ we have 
\[
(1-\tilde{\psi})(T_{1}\circ J_{+}(.)_{\eta}\circ I_{-}(.)_{\eta}\circ T_{1})^{n}(\varrho(.,\lambda))=O_{\mathcal{S}(\mathbb{R})}(\tilde{\lambda}^{-\infty})
\]
we use the same arguments as in the second part of the proof of Lemma \ref{lemlocal}.\end{itemize}
\end{proof}
\begin{dfn}\label{dfnmodvrn}
Let $\varrho(.,\lambda)\in\mathcal{S}_{K_{0}}(\lambda)$ and $\eta\in \text{supp}(\Psi)$. For $1\leq n\leq N$, $N\simeq\lambda h^{\epsilon}$ set
\[
\varrho^{n}(z,\eta,\lambda)=(-1)^{n}\Psi(\eta)(T_{1}\circ J_{+}(.)_{\eta}\circ I_{-}(.)_{\eta}T_{1})^{n}(\varrho(.,\lambda))(z),
\]
\[
\varrho^{0}(z,\eta,\lambda)=\varrho(z,\lambda)\Psi(\eta).
\]
%and for negative values $n\in\{-2,-1\}$ set
%\[
%\varrho^{n}(z,\eta,\lambda)=(-1)^{n}\Psi(\eta)(T^{-1}\circ J_{+}(.)_{\eta}\circ I_{-}(.)_{\eta})^{n}(\varrho(.,\lambda))(z).
%\]

From Proposition \ref{propimportant} it follows that $\varrho^{n}(z,\eta,\lambda)\in\mathcal{S}_{K_{0}}(\lambda/n)$. Let
\begin{equation}\label{forgn}
g^{n}(t,s,\eta,h)=\varrho^{n}(\frac{t+2(1+a)^{1/2}s}{2(1+a)^{1/2}a^{1/2}}-2n,\eta,\lambda),
\end{equation}
\begin{equation}
\Phi_{n}(t,x,y,s,\eta)=\eta(y-t(1+a)^{1/2}+(x-a)s+\frac{s^{3}}{3}+\frac{4}{3}n a^{3/2}),
\end{equation}
and set
\begin{equation}\label{foru}
U_{h}=\sum_{n=0}^{N}u^{n}_{h}, \quad u^{n}_{h}(t,x,y)=\int e^{\frac{i}{h}\Phi_{n}}g^{n}(t,s,\eta,h)ds d\eta.
\end{equation}
\end{dfn}
\begin{prop}\label{propestdirbound}
With this choice of the symbols $g^{n}$ we have for all $0\leq n\leq N-1$
\begin{equation}
Tr_{-}(u^{n}_{h})(t,y;h)+Tr_{+}(u^{n+1}_{h})(t,y;h)=O_{L^{2}}(\lambda^{-\infty}).
\end{equation} 
\end{prop}
\begin{proof}
The proof follows from Propositions \ref{bound}, \ref{propimportant} and the definition of the symbols $g^{n}$, since we have
\[
e^{\frac{i\pi}{2}}I_{-}(T_{1}(\varrho^{n}(.,\eta,\lambda)))_{\eta}+e^{-\frac{i\pi}{2}}I_{+}(T_{-1}(\varrho^{n+1}(.,\eta,\lambda)))_{\eta}=O_{\mathcal{S}(\mathbb{R})}(\lambda^{-\infty}),
\]
where $T_{\pm 1}$ are the translation operators which to a given $\varrho(z)$ associate $\varrho(z\pm 1)$ and since the operators $I_{\pm,\eta}$ are of convolution type so they commute with translations.
\end{proof}

\section{Strichartz estimates for the approximate solution} 
Let $U_{h}$ be
given by \eqref{foru} and let $(q,r)$ a sharp wave-admissible pair
in dimension $2$, i.e. such that $\frac{1}{q}=\frac{1}{2}(\frac{1}{2}-\frac{1}{r})$. The "cusp" is reflected with period $a^{1/2}$, $a=h^{\delta}$ and in order
to compute the norm of $U_{h}$ on a finite
interval of time we will take $\delta=\frac{1-\epsilon}{2}$ in order to obtain
\begin{equation}
1\simeq Na^{1/2}\simeq\lambda h^{\epsilon}h^{\delta/2}.
\end{equation}
We prove the following
\begin{prop}\label{propcountexstrichartz}
Let $r>4$, $\beta(r)=\frac{3}{2}(\frac{1}{2}-\frac{1}{r})+\frac{1}{6}(\frac{1}{4}-\frac{1}{r})$ and let $\beta\leq\beta(r)-\epsilon$ for $\epsilon>0$ small enough like before. Then the approximate solution of the wave equation \eqref{wavvv} satisfies 
\begin{equation}\label{stricon}
h^{\beta}\|U_{h}\|_{L^{q}([0,1],L^{r}(\Omega))}\gg
\|U_{h}|_{t=0}\|_{L^{2}(\Omega)}.
\end{equation}
In particular, the restriction on $\beta$ shows that the Strichartz inequalities of the free case are not valid, there is a loss of at least $\frac{1}{6}(\frac{1}{4}-\frac{1}{r})$ derivatives.
\end{prop}
\begin{proof}
In the construction of $U_{h}$ we considered an initial "cusp" $u^{0}_{h}$ of the form \eqref{integral}, with symbol $g$ given by \eqref{gforme},
with $\varrho^{0}\in\mathcal{S}_{[-c_{0},c_{0}]}(a^{3/2}/h)$ depending only on the integral curves of the vector field $Z$ and $\eta$, supported for $\eta$ in a small neighborhood of $1$. We introduce the Lagrangian manifold associated to $u^{n}_{h}$, with phase function $\Phi_{n}=\Phi+\frac{4}{3}n\eta a^{3/2}$
\begin{multline}\label{lagran}
\Lambda_{\Phi_{n}}:=\{(t,x,y,\tau=-(1+a)^{1/2}\eta,\xi=s\eta,\eta),\\ a-x=s^{2}, y-t(1+a)^{1/2}+\frac{4}{3}na^{3/2}=\frac{2}{3}s^{3}\}\subset T^{*}\mathbb{R}^{3}.
\end{multline}
\begin{lemma}\label{lemlagramodwfset}
Let $u^{n}_{h}$ be given by \eqref{foru}, then $WF_{h}(u^{n}_{h})\subset\Lambda_{\Phi_{n}}$.
\end{lemma}
\begin{proof}
If $|\partial_{s}\Phi_{n}|\geq c>0$ we use the operator $L_{1}=\frac{h}{i|s^{2}+a-x|}\partial_{s}$ in order to gain a power of $h^{1-\frac{\delta}{2}}$ at each integration by parts with respect to $s$,  thus the contribution we get in this case is $O_{L^{2}}(h^{\infty})$.
Let now $|\partial_{\eta}\Phi_{n}|\geq c>0$ for some positive constant $c$:  before making (repeated) integrations by parts using this time the operator $L_{2}=\frac{h\partial_{\eta}\Phi_{n}}{i|\partial_{\eta}\Phi_{n}|^{2}}\partial_{\eta}$ we need to estimate the derivatives with respect to $\eta$ for each $g^{n}$ defined in \eqref{forgn}. We have
\begin{multline}
u^{n}_{h}(t,x,y)=(-1)^{m}\int e^{\frac{i}{h}\Phi_{n}}L^{*m}_{2}(\varrho^{n}(\frac{t+2(1+a)^{1/2}s}{2(1+a)^{1/2}a^{1/2}}+2n,\eta,\lambda))d\eta ds\\
=(-1)^{m+n}h^{m}\int e^{\frac{i}{h}\Phi_{n}}\frac{(\partial_{\eta}\Phi_{n})^{m}}{|\partial_{\eta}\Phi_{n}|^{2m}}\Big(\sum_{k=0}^{m}\partial^{m-k}_{\eta}\Psi(\eta)\partial^{k}_{\eta}(F_{\eta\lambda})^{*n}\Big)*\\ *\varrho(,\lambda)(\frac{t+2(1+a)^{1/2}s}{2(1+a)^{1/2}a^{1/2}}+2n)d\eta ds,
\end{multline}
where $*$ denotes the convolution product. The derivatives of $(F_{\eta\lambda})^{*n}$ with respect to $\eta$ are easily computed using the explicit form of $(F_{\eta\lambda})^{*n}$ that we recall (see the proof of Proposition \ref{propimportant}):
\[
(F_{\eta\lambda})^{*n}(z)=\frac{\eta\tilde{\lambda}}{2\pi}\int_{\tilde{\zeta}}e^{i\eta\tilde{\lambda}(z\tilde{\zeta}+ n^{2}f(\frac{\tilde{\zeta}}{n}))}c^{n}(\frac{\tilde{\zeta}}{\lambda},\eta n\tilde{\lambda})d\tilde{\zeta},
\]
with $c(\zeta,\omega)=\chi^{2}(\zeta)\sum_{j\geq 0}c_{j}(1-\zeta)^{-3j/2}\omega^{-j}$ and where we have made the change of variable $\tilde{\zeta}=n\zeta$ and set $\tilde{\lambda}=\lambda/n\geq h^{-\epsilon}\gg 1$. Hence one $\eta$-derivative yields
\begin{multline}\label{derivfetlan}
\partial_{\eta}(F_{\eta\lambda})^{*n}(z)=\frac{1}{\eta}(F_{\eta\lambda})^{*n}(z)+\frac{\eta\tilde{\lambda}}{2\pi}\int_{\tilde{\zeta}}e^{i\eta\tilde{\lambda}(z\tilde{\zeta}+n^{2}f(\frac{\tilde{\zeta}}{n}))}i\tilde{\lambda}(z\tilde{\zeta}+ n^{2}f(\frac{\tilde{\zeta}}{n}))c^{n}(\frac{\tilde{\zeta}}{n},\eta n\tilde{\lambda})d\tilde{\zeta}\\
+ \frac{\eta\tilde{\lambda}}{2\pi}\int_{\tilde{\zeta}}e^{i\eta\tilde{\lambda}(z\tilde{\zeta}+ n^{2}f(\frac{\tilde{\zeta}}{n}))}n\partial_{\eta}c(\frac{\tilde{\zeta}}{n},\eta n \tilde{\lambda})c^{n-1}(\frac{\tilde{\zeta}}{n},\eta n \tilde{\lambda})d\tilde{\zeta}.
\end{multline}
The symbol of the third term in the right hand side of \eqref{derivfetlan} is $n\partial_{\eta}c(\zeta,\eta\lambda)c^{n-1}(\zeta,\eta\lambda)$ and we have
\[
\partial_{\eta}c(\zeta,\eta\lambda)=-\eta^{-2}\lambda^{-1}\sum_{j\geq 1}jc_{j}(1-\zeta)^{-3j/2}(\eta\lambda)^{-(j-1)},
\]
and since $n\ll \lambda$, the contribution from this term is easily handled with.

The symbol in the second term in the right hand side of \eqref{derivfetlan} equals the symbol of $(F_{\eta\lambda})^{*n}$ multiplied by the factor $i\tilde{\lambda}(z\tilde{\zeta}+\lambda n f(\frac{\tilde{\zeta}}{n}))$. Recall that on the support of $c(\zeta,\eta \lambda)$ we have $\zeta=\tilde{\zeta}/n\in \text{supp}(\chi)$ is as close to zero as we want and there $f(\zeta)=\zeta^{2}/2+O(\zeta^{3})$, hence
$n^{2}f(\frac{\tilde{\zeta}}{n})=\tilde{\zeta}^{2}/2+O(\tilde{\zeta}^{3}/n)$.
On the other hand, when we take the convolution product of the second term in \eqref{derivfetlan} with $\varrho^{0}(.,\lambda)$ we obtain in the same way as in the proof of Proposition \ref{propimportant} that the critical points of the phase in the oscillatory integral obtained in this way,
\[
\frac{\eta\tilde{\lambda}}{2\pi}\int_{\tilde{\zeta},z'}e^{i\eta\tilde{\lambda}((z-z')\tilde{\zeta}+n^{2}f(\frac{\tilde{\zeta}}{n}))}i\tilde{\lambda}((z-z')\tilde{\zeta}+n^{2}f(\frac{\tilde{\zeta}}{n}))c^{n}(\frac{\tilde{\zeta}}{n},\eta n\tilde{\lambda})\varrho^{0}(z',\lambda)d\tilde{\zeta}dz',
\]
are given by $\tilde{\zeta}=0$ and $z=z'$. The phase function which will be denoted again by $\phi_{n}(z,z',\tilde{\zeta})$ as before satisfies $\phi_{n}(z,z,0)=0$, $\partial_{z'}\phi_{n}(z,z,0)=0$ and $\partial_{\tilde{\zeta}}\phi_{n}(z,z,0)=0$. Applying the stationary phase theorem in $\tilde{\zeta}$ and $z'$, the first term in the asymptotic expansion obtained in this way vanishes, and the next ones are multiplied by strictly negative, integer powers of $\tilde{\lambda}$, hence the contribution from this term will is also bounded. 

Notice that when we take higher order derivatives in $\eta$ of $\varrho^{n}$, we obtain symbols which are products of $\tilde{\lambda}^{j} (\phi_{n})^{j}\partial^{k-j}_{\eta}(c^{n}(\tilde{\zeta}/n,\eta n \tilde{\lambda}))$ and can be dealt with in the same way, taking into account this time that the first $j$ terms in the asymptotic expansion obtained after applying the stationary phase vanish. As a consequence, after each integration by parts in $\eta$ using the operator $L_{2}$ we gain a factor $h$,  meaning that the contribution of $u^{n}_{h}$ is $O_{L^{2}}(h^{\infty})$.
 \end{proof}
We also need the next results:
\begin{lemma}\label{suppo}
If $\varrho(.,\lambda)\in\mathcal{S}_{[-c_{0},c_{0}]}(\lambda)$ with $0<c_{0}<1$ sufficiently small, then $u^{n}_{h}$ have almost disjoint supports in the time variable $t$.
\end{lemma}
\begin{proof}
Let $\mu\in(0,1)$ and $|t-4n(1+a)^{1/2}a^{1/2}|\geq 2(1+a)^{1/2}a^{1/2}(1+\mu)$. Then on the essential support of $\varrho^{n}(\frac{t+2(1+a)^{1/2}s}{2(1+a)^{1/2}a^{1/2}}-2n,\eta,\lambda)$ we must have $|s|\geq a^{1/2}(1+\mu-c_{0})$ while on the Lagrangian $\Lambda_{\Phi_{n}}$ defined in \eqref{lagran} we have $|a-x|=s^{2}\leq a$. Consequently, if $\mu\geq c_{0}+\epsilon_{0}$ for some $\epsilon_{0}>0$ as small as we want, we are not anymore on the Lagrangian $\Lambda_{\Phi_{n}}$. Since outside any neighborhood of $\Lambda_{\Phi_{n}}$ the contribution in the integral defining $u^{n}_{h}$ is $O_{L^{2}}(h^{\infty})$, we conclude that $u^{n}_{h}$ "lives" essentially on a  time interval 
\[
[4n(1+a)^{1/2}a^{1/2}-2(1+a)^{1/2}a^{1/2}(1+c_{0}),4n(1+a)^{1/2}a^{1/2}+2(1+a)^{1/2}a^{1/2}(1+c_{0})].
\] 
\end{proof}
Since $a=h^{\delta}\ll 1$ and therefor $(1+a)^{1/2}\simeq 1$ we claim that $u^{n}_{h}$ is in fact essentially supported for $t$ in the time interval
\[
[4na^{1/2}-2a^{1/2}(1+c_{0}),4na^{1/2}+2a^{1/2}(1+c_{0})]
\]

\begin{lemma}\label{lemunic}
Let $0<c_{0}<1/3$ and let $I_{k}$ be small neighborhoods of  $4a^{1/2}k$ of size $a^{1/2}$,
\[
I_{k}=[4ka^{1/2}-a^{1/2}c_{0},4ka^{1/2}+a^{1/2}c_{0}].
\]
If $t\in I_{k}$ then in the sum $U_{h}(t,.)$ there is only one cusp that appears, $u^{k}_{h}(t,.)$, the contribution from all the others $u^{n}_{h}(t,.)$ with $n\neq k$ being $O_{L^{2}}(h^{\infty})$. 
\end{lemma}
\begin{proof}
On the essential support of $\varrho^{n}(.,\eta,\lambda)$ one has 
\[
|t+2(1+a)^{1/2}s-4n(1+a)^{1/2}a^{1/2}|\leq 2(1+a)^{1/2}a^{1/2}c_{0}.
\]
Suppose $n\neq k$: we have to show that the contribution from $u^{n}_{h}$ is $O_{L^{2}}(h^{\infty})$. Write
\[
2(1+a)^{1/2}a^{1/2}c_{0}\geq 4|n-k|(1+a)^{1/2}a^{1/2}-|t-4k(1+a)^{1/2}a^{1/2}|-2(1+a)^{1/2}|s|
\]
\[
\geq 4(1+a)^{1/2}a^{1/2}-(1+a)^{1/2}a^{1/2}c_{0}-2(1+a)^{1/2}|s|,
\]
which yields $|s|\geq 3a^{1/2}/2$ since $c_{0}<1/3$ and as in the proof of Lemma \ref{suppo} we see that we are localized away from a neighborhood of $\Lambda_{\Phi_{n}}$ (on which $|s|\leq a^{1/2}$), thus the contribution is $O_{L^{2}}(h^{\infty})$. Consequently, the only nontrivial part comes from $n=k$ in which case we find $|s|\leq 3c_{0}a^{1/2}/2\leq a^{1/2}/2$, thus the $k$-th "piece of cusp" does not reach the boundary $\{x=0\}$ (since on the Lagrangian $\Lambda_{\Phi_{k}}$ we have $a-x=s^{2}$ and outside any neighborhood of $\Lambda_{\Phi_{k}}$ the contribution is $O_{L^{2}}(h^{\infty})$).
\end{proof}

We turn to the proof of Proposition \ref{propcountexstrichartz}. We use Lemma \ref{lemunic} and Proposition \ref{propnorm} from the Appendix to estimate from below the $L^{q}([0,1],L^{r}(\Omega))$ norm of $U_{h}$:
\begin{align}\label{bonor}
\|U_{h}\|^{q}_{L^{q}([0,1],L^{r}(\Omega))} & =\int_{0}^{1}\|U_{h}\|^{q}_{L^{r}(\Omega)}dt =
\int_{0}^{1}\|\sum_{n=0}^{N}u^{n}_{h}\|^{q}_{L^{r}(\Omega)}dt \\
& \geq \sum_{k\leq N/5}\int_{t\in I_{k}}\|\sum_{n=0}^{N}u^{n}_{h}\|^{q}_{L^{r}(\Omega)}dt +O(h^{\infty})\\ & \simeq \sum_{k\leq N/5}|I_{k}|\|u^{0}_{h}\|^{q}_{L^{r}(\Omega)} +O(h^{\infty})\\
& \simeq \|u^{0}_{h}\|^{q}_{L^{r}(\Omega)}+O(h^{\infty}).
\end{align}
Indeed, we have shown in Lemma \ref{lemunic} that for $t$ belonging to sufficiently small intervals of time $I_{k}$ there is only $u^{k}_{h}$ to be considered in the sum since the supports of $u^{n}_{h}$ will be disjoints. On the other hand, for $t\in I_{k}$, $u^{k}_{h}(t,.)$ admits a cusp singularity at $x=a$ which guarantees that the piece of cusp does not "live" enough to reach the boundary. Moreover, we see from Proposition \ref{propnorm} that for $t\in I_{k}$ the $L^{r}(\Omega)$ norms of $u^{k}_{h}(t,.)$ are equivalent to the $L^{r}(\Omega)$ norms of $u^{0}_{h}$. Using Corollary \ref{corn} we deduce that there are constants $C$ independent of $h$ such that for $r=2$ 
\begin{equation}\label{estnorm21}
\|U_{h}|_{t=0}\|_{L^{2}(\Omega)}=\|u_{h}|_{t=0}\|_{L^{2}(\Omega)}\simeq
h^{1+\frac{\delta}{4}},
\end{equation}
while for $r>4$ 
\begin{equation}\label{estnorm3}
\|U_{h}\|_{L^{q}([0,1],L^{r}(\Omega))}\geq C
h^{\frac{1}{3}+\frac{5}{3r}}
\end{equation}
and since $\delta=\frac{1-\epsilon}{2}$ we deduce that \eqref{stricon} holds for $\beta\leq\beta(r)-\epsilon$ since we have 
\begin{multline}
h^{\beta}\|U_{h}\|_{L^{q}([0,1],L^{r}(\Omega))}\geq C h^{\beta(r)-\epsilon}h^{\frac{1}{3}+\frac{5}{3r}}
=Ch^{-7\epsilon/8} h^{1+(1-\epsilon)/8}\\\gg h^{1+\frac{\delta}{4}}\simeq \|U_{h}|_{t=0}\|_{L^{2}(\Omega)}.
\end{multline}
\begin{rmq}
Notice that for $2\leq r<4$ 
\begin{equation}\label{estimnorm11}
\|U_{h}\|_{L^{q}([0,1],L^{r}(\Omega))}\geq C
h^{\frac{1}{r}+\frac{1}{2}+\delta(\frac{1}{r}-\frac{1}{4})},
\end{equation}
therefor in this case the previous construction doesn't provide a contradiction
to the Strichartz inequalities when compared to the free case.
\end{rmq}
\end{proof}

\begin{prop}
The approximate solution $U_{h}$ defined in \eqref{foru} satisfies the Dirichlet boundary condition
\begin{equation}\label{diruh01}
U_{h}|_{[0,1]\times \partial\Omega}=O(h^{\infty}).
\end{equation}
\end{prop}
\begin{proof}
Using Propositions \ref{bound} and \ref{propestdirbound}, the contribution of $U_{h}$ on the boundary writes
\begin{equation}\label{diruhohinf}
U_{h}(t,0,y) =\sum_{n=0}^{N}\sum_{\pm}Tr_{\pm}(u^{n}_{h})(t,y;h)\\  =Tr_{+}(u^{0}_{h})(t,y;h)+Tr_{-}(u^{N}_{h})(t,y;h).
\end{equation}
The first term in the right hand side of \eqref{diruhohinf} is easy to handle  since 
$Tr_{+}(u^{0}_{h})(t,y;h)$ is essentially supported for 
\[
t\in [-2(1+c_{0})a^{1/2},-2(1-c_{0})a^{1/2}].
\]
Since we consider only the restriction to $[0,1]\times\partial\Omega$, the contribution from this term will be $O_{L^{2}}(h^{\infty})$. To deal with the second term in the right hand side of \eqref{diruhohinf} we first study the essential support of $Tr_{-}(u^{N}_{h})(t,y;h)$ for $t\in [0,1]$. We distinguish two situations:
\begin{itemize}
\item If $(4h^{-\delta/2})^{-1}-[(4h^{-\delta/2})^{-1}]<1/2$, where we denoted by $[z]$ the integer part of $z$ we take
\[
N:=[(4h^{-\delta/2})^{-1}] 
\]
and we deduce that  $Tr_{-}(u^{N}_{h})(t,y;h)$ is essentially supported for $t$ in an interval strictly contained in $[0,1]$ while $Tr_{+}(u^{N}_{h})(t,y;h)$ has a nontrivial contribution only on
\[
[4Na^{1/2}-2(1+c_{0})a^{1/2},4Na^{1/2}-2(1-c_{0})a^{1/2}].
\]
A direct computation shows that for this choice of $N$
\[
4Na^{1/2}-2(1+c_{0})a^{1/2} \simeq 4h^{-\delta}[(4h^{-\delta})^{-1}]+\frac{1}{2}(4h^{-\delta})^{-1}> 1.
\]
Therefor, on $[0,1]$ the contribution of $Tr_{+}(u^{N}_{h})(t,y;h)$ is canceled by $Tr_{-}(u^{N-1}_{h})(t,y;h)$, while the contribution of $Tr_{-}(u^{N}_{h})$ equals $O_{L^{2}}(h^{\infty})$ since it is essentially supported outside $[0,1]$.
 
\item If $(4h^{-\delta/2})^{-1}-[(4h^{-\delta/2})^{-1}]\geq1/2$, we set
\[
N:=[(4h^{-\alpha/2})^{-1}] +1
\]
and we conclude using the same arguments as in the preceding case.
\end{itemize}

\end{proof}

\section{End of the proof of Theorem \ref{thm1}}
Let $U_{h}$ be the approximate solution to the wave equation \eqref{ondes} defined by \eqref{foru}. In \eqref{estnorm21} we obtained $\|U_{h}|_{t=0}\|_{L^{2}(\Omega)}\simeq h^{1+\delta/4}$. We now consider the $L^{2}$-normalized approximate solution $W_{h}=\frac{1}{\|U_{h}|_{t=0}\|_{L^{2}(\Omega)}}U_{h}$. We also let $V_{h}=W_{h}+w_{h}$, where $V_{h}$ solves
\begin{equation}
\square V_{h}=0,\quad V_{h}|_{[0,1]\times\partial\Omega}=0,
\end{equation}
with initial data
\begin{equation}
V_{h}|_{t=0}=W_{h}|_{t=0},\quad \partial_{t}V_{h}|_{t=0}=\partial_{t}W_{h}|_{t=0}.
\end{equation}
\begin{prop}\label{proper}
Under the preceding assumptions $w_{h} $ satisfies 
\begin{equation}\label{err}
\|\square w_{h}\|_{L^{2}(t\in [0,1],L^{2}(\Omega))}=O(h^{-\delta}),\quad w_{h}|_{\partial\Omega}=O_{L^{2}}(h^{\infty}),
\end{equation}
\begin{equation}
w_{h}|_{t=0}=0,\quad \partial_{t}w_{h}|_{t=0}=0.
\end{equation}
\end{prop}
\begin{proof}
If we set $\alpha(h):=\|U_{h}|_{t=0}\|_{L^{2}(\Omega)}\simeq h^{1+\delta/4}$, then one has
\begin{align}
\|\square w_{h}\|^{2}_{L^{2}(t\in [0,1],L^{2}(\Omega))} & =\alpha(h)^{-2}\|\sum_{n=0}^{N}\square u^{n}_{h}\|^{2}_{L^{2}(t\in [0,1])L^{2}(\Omega)}\\
& \lesssim  \alpha(h)^{-2}\sum_{k\leq N/4}\int_{J_{k}}\|\sum_{n=0}^{N}\square u^{n}_{h}\|^{2}_{L^{2}(\Omega)}dt+O(h^{\infty}) \\
& \lesssim 8 \alpha(h)^{-2}\sum_{k\leq N/4}\int_{J_{k}}\|\square u^{k}_{h}\|^{2}_{L^{2}(\Omega)}+O(h^{\infty}),
\end{align}
where 
\[
J_{k}:=[4a^{1/2}k-2a^{1/2},4a^{1/2}k+2a^{1/2}],
\]
and where we used the fact that for each $n$ there are at most three cusps to consider for $t\in J_{k}$ as shown in Lemma \ref{lemunic}. Let us estimate $\|\square u^{k}_{h}(t,.)\|_{L^{2}(\Omega)}$ for $t\in J_{k}$. The proof of Proposition \ref{propnorm} of the Appendix applied to $\square u^{k}_{h}$ (computed in \eqref{boxx}) yields
\[
\|\square u^{k}_{h}(t,.)\|_{L^{2}(\Omega)}\lesssim h^{-\delta+1+\delta/4},
\]
since the assumption $\varrho\in\mathcal{S}_{[-c_{0},c_{0}]}(\lambda)$ implies that $\sup_{z}|\partial^{2}_{z}\varrho|\leq C$ for some constant $C$ independent of $\lambda$ and one can bound from above the $L^{2}(\Omega)$ norm of $\square u^{k}_{h}$ (notice that the only difference between the estimates concerning $u^{k}_{h}$ is that instead of $\varrho^{n}$ we now have $\partial^{2}\varrho^{n}$ which we handle in the same way).
Consequently we obtain
\[
\|\square w_{h}\|^{2}_{L^{2}(t\in[0,1],L^{2}(\Omega))}\lesssim a(h)^{-2}\sum_{k\leq N/4}|J_{k}|h^{-2\delta+2+\delta/2}
\lesssim h^{-2\delta},
\]
since $|J_{k}|$ are of size $a^{1/2}$, $k\leq N/4$ and $Na^{1/2}\simeq 1$.
\end{proof}
\begin{cor}\label{cor}
If $(q,r)$ is a sharp wave-admissible pair in dimension two then $w_{h}$ satisfies
\begin{equation}
\|w_{h}\|_{L^{q}([0,1],L^{r}(\Omega))}\leq Ch^{1-\delta-2(\frac{1}{2}-\frac{1}{r})},
\end{equation}
where $C$ is some constant independent of $h$.
\end{cor}
\begin{proof}
Write the Duhamel formula for $w_{h}$,
\begin{equation}
w_{h}(t,x,y)=\int_{0}^{t}\frac{\sin(t-\tau)\sqrt{-\Delta_{D}}}{\sqrt{-\Delta_{D}}}(\square w_{h}(\tau,.))d\tau.
\end{equation}
Using the Minkowsky inequality and Proposition \ref{proper} we find
\begin{multline}\label{esterr}
\|w_{h}\|_{L^{\infty}([0,1],H^{1}(\Omega))}=\|\int_{0}^{t}\frac{\sin(t-\tau)\sqrt{-\Delta_{D}}}{\sqrt{-\Delta_{D}}}(\square w_{h}(\tau,.))d\tau\|_{L^{\infty}([0,1],H^{1}(\Omega))}\\
\leq\int_{0}^{1} \|\frac{\square w_{h}(\tau,.)}{\sqrt{-\Delta_{D}}}\|_{H^{1}(\Omega)}d\tau\simeq \|\square w_{h}\|_{L^{1}([0,1],L^{2}(\Omega))}\leq Ch^{-\delta}.
\end{multline}
\begin{rmq} 
Notice that since we are dealing with the Dirichlet Laplace operator $\Delta_{D}$ inside a bounded domain there is no problem in estimating $\|(\sqrt{-\Delta_{D}})^{-1}f\|_{H^{1}(\Omega)}$ by $\|f\|_{L^{2}(\Omega)}$. Indeed, let $(e_{\nu_{j}})_{j\geq 0}$ be the eigenbasis of $L^{2}(\Omega)$ consisting in eigenfunctions of $-\Delta_{D}$ associated to the eigenvalues $\nu^{2}_{j}$ considered in non-decreasing order and decompose $f=\sum_{j\geq 0}f_{j}e_{\nu_{j}}$, $f_{j}=<f,e_{\nu_{j}}>$. Then
\[
(\sqrt{-\Delta_{D}})^{-1}f\simeq \sum_{j}\frac{1}{\nu_{j}}f_{j}e_{\nu_{j}}
\]
and since $\nu_{1}\geq c>0$ for some fixed constant $c>0$ we can estimate
\[
\|(\sqrt{-\Delta_{D}})^{-1}f\|^{2}_{H^{1}(\Omega)}\simeq \sum_{j\geq 0}\frac{(1+\nu^{2}_{j})}{\nu^{2}_{j}}\|f_{j}\|^{2}_{L^{2}(\Omega)}.
\]
Take now $C=\sup_{j}(1+1/\nu^{2}_{j})\leq 1+1/c^{2}$, then
\begin{equation}\label{h1l2}
\|(\sqrt{-\Delta_{D}})^{-1}f\|_{H^{1}(\Omega)}\leq \sqrt{C}\|f\|_{L^{2}(\Omega)}.
\end{equation}
If, instead, we were considering the Neumann Laplacian $\Delta_{N}$ inside the domain $\Omega$, in order to obtain bounds like in \eqref{h1l2} we had to introduce a cut-off function $\Psi\in C^{\infty}_{0}(\mathbb{R})$ equal to $1$ close to $0$ and decompose a function $f$
\[
f=\Psi(-\Delta_{N})f+(1-\Psi(-\Delta_{N}))f
\]
and treat separately the contribution $\Psi(-\Delta_{N})f$ obtained for small frequencies of $f$.
\end{rmq}
In order to obtain estimates for the $L^{\infty}([0,1],L^{r}(\Omega))$ norms of $w_{h}$ we also need to establish bounds from above for its $L^{\infty}([0,1],L^{2}(\Omega))$ norms. We need the next result:
\begin{prop}\label{proph}
Let $f(x,y):\Omega\rightarrow\mathbb{R}$ be localized at frequency $1/h$ in the $y\in\mathbb{R}^{d-1}$ variable, i.e. such that there exists $\psi\in C^{\infty}_{0}(\mathbb{R}^{d-1}\setminus 0)$ with $\psi(hD_{y})f=f$. Then there exists a constant $C>0$ independent of $h$ such that one has
\[
\|f\|_{H^{-1}(\Omega)}\leq Ch\|f\|_{L^{2}(\Omega)}.
\]
\end{prop}
\begin{proof}
Since $\chi(hD_{y})f=f$ we have
\[
\|f\|_{H^{-1}(\Omega)}=\sup_{\|g\|_{H^{1}(\Omega)}\leq 1}\int \psi f\bar{g}\leq\|f\|_{L^{2}(\Omega)}\times \sup_{\|g\|_{H^{1}(\Omega)}\leq 1}\|\psi(hD_{y})g\|_{L^{2}(\Omega)}
\]
\[
\leq h\|f\|_{L^{2}(\Omega)}\|\tilde{\psi}(hD_{y})\nabla_{y}g\|_{L^{2}(\Omega)}\leq C h\|f\|_{L^{2}(\Omega)},
\]
where we set $\tilde{\psi}(\eta)=|\eta|^{-1}\psi(\eta)$.
\end{proof}
Using again Duhamel's formula written above, we have
\begin{equation}
\|w_{h}\|_{L^{\infty}([0,1],L^{2}(\Omega))}\lesssim\|\square w_{h}\|_{L^{1}([0,1],H^{-1}(\Omega))}
\end{equation}
and from Proposition \ref{proph} applied to $f=\square w_{h}$ we deduce
\begin{equation}\label{err2}
\|w_{h}\|_{L^{\infty}([0,1],L^{2}(\Omega))}\lesssim h\|\square w_{h}\|_{L^{1}([0,1],L^{2})}\lesssim Ch^{1-\delta}.
\end{equation}
Interpolation between \eqref{esterr} and \eqref{err2} with weights $\sigma$ and $1-\sigma$ yields
\begin{equation}
\|w_{h}\|_{L^{\infty}([0,1],H^{\sigma}(\Omega))}\leq Ch^{1-\delta-\sigma}.
\end{equation}
We take $\sigma=2(\frac{1}{2}-\frac{1}{r})$ and use the Sobolev inequality in order to obtain
\begin{equation}
\|w_{h}\|_{L^{q}([0,1],L^{r}(\Omega))}\leq Ch^{1-\delta-2(\frac{1}{2}-\frac{1}{r})}.
\end{equation}
\end{proof}
\emph{End of the proof of Theorem} \ref{thm1}

From Corollary \ref{cor} we see that the norm $\|w_{h}\|_{L^{q}(([0,1],L^{r}(\Omega))}$ is much smaller then the norm of $\|W_{h}\|_{L^{q}(([0,1],L^{r}(\Omega))}$: in fact we have to check that the following inequality holds for $r>4$
\begin{equation}\label{label}
h^{1-\delta-2(\frac{1}{2}-\frac{1}{r})}\ll h^{\frac{1}{3}+\frac{5}{3r}-1-\frac{\delta}{4}}
\end{equation}
which is  obviously true. 
Let $\beta<\beta(r)=\frac{3}{2}(\frac{1}{2}-\frac{1}{r})+\frac{1}{6}(\frac{1}{4}-\frac{1}{r})$. We have
\begin{align}
h^{\beta}\|V_{h}\|_{L^{q}([0,1],L^{r}(\Omega))} & \geq h^{\beta}(\|W_{h}\|_{L^{q}([0,1],L^{r}(\Omega))}-\|w_{h}\|_{L^{q}([0,1],L^{r}(\Omega))})\\
& \geq \frac{1}{2}h^{\beta}\|W_{h}\|_{L^{q}([0,1],L^{r}(\Omega))}\gg 1.
\end{align}
On the other hand $\|V_{h}\|_{L^{2}(\Omega)}\simeq 1$,  $h\|\partial_{t}V_{h}|_{t=0}\|_{L^{2}(\Omega)}\simeq 1$ , thus  for $\beta<\beta(r)$ the (exact) solution $V_{h}$ satisfies
\begin{equation}
h^{\beta}\|V_{h}\|_{L^{q}([0,1],L^{r}(\Omega))}\gg
\|V_{h}|_{t=0}\|_{L^{2}(\Omega)}.
\end{equation}
The proof of Theorem \ref{thm1} is complete.

\section{Appendix}

\subsection{Proof of Lemma \ref{lem1} ( $TT^{*}$ argument)}\label{ttstar}
\begin{proof}
Let $0<T_{0}<\infty$ and denote by $T$ the operator which to a
given $u_{0}\in L^{2}(\mathbb{R}^{n})$ associates $U(t)\psi(hD)u_{0}\in
L^{q}([0,T_{0}],L^{r}(\mathbb{R}^{n}))$, where by $U(t)=e^{-\frac{it}{h}G}$ we denoted the linear flow. Its adjoint
$T^{*}:L^{q'}([0,T_{0}],L^{r'}(\mathbb{R}^{n}))\rightarrow L^{2}(\mathbb{R}^{n})$ is given by
\begin{equation}
(T^{*}g)(x)=\int_{0}^{T_{0}}\psi^{*}U(-t)g(t,x)dt
\end{equation}
thus we can write
\begin{equation}\label{tt}
(TT^{*}g)(t,x)=\int_{0}^{T_{0}}U(t)\psi\psi^{*}U(-s)g(s,x)ds=\int_{0}^{T_{0}}U(t-s)\psi\psi^{*}g(s,x)ds
\end{equation}
since $\psi$ has constant coefficients. Suppose that the dispersive
estimate 
\begin{equation}\label{dispersion2}
\|e^{-\frac{it}{h}G}\psi(hD)u_{0}\|_{L^{\infty}(\mathbb{R}^{n})}\lesssim (2\pi
h)^{-n}\gamma_{n,h}(\frac{t}{h})\|\psi(hD)u_{0}\|_{L^{1}(\mathbb{R}^{n})}
\end{equation}
holds for a function $\gamma_{n,h}:\mathbb{R}\rightarrow\mathbb{R}_{+}$.
Interpolation between \eqref{dispersion2} and the energy estimates
gives
\begin{equation}\label{rr}
\|e^{-\frac{it}{h}G}\psi(hD)u_{0}\|_{L^{r}(\mathbb{R}^{n})}\leq
Ch^{-n(1-\frac{2}{r})}\gamma_{n,h}(\frac{t}{h})^{1-\frac{2}{r}}\|u_{0}\|_{L^{r'}(\mathbb{R}^{n})},
\end{equation}
and from \eqref{tt} and \eqref{rr} we deduce
\begin{equation}
\|TT^{*}\|_{L^{q}((0,T_{0}],L^{r}(\mathbb{R}^{n}))}\leq
Ch^{-n(1-\frac{2}{r})}\|\int_{0}^{T_{0}}\gamma_{n,h}(\frac{t-s}{h})^{1-\frac{2}{r}}\|g(s)\|
_{L^{r'}(\mathbb{R}^{n})}ds\|_{L^{q}[0,T_{0}]}.
\end{equation}
The application $|t|^{-\frac{2}{q}}\ast :
L^{q'}\rightarrow L^{q}$ is bounded for $q>2$ by Hardy-Littlewood-Sobolev theorem, thus we obtain
\eqref{estim},
\begin{equation}
\|T\|^{2}_{L^{2}\rightarrow L^{q}((0,T_{0}],L^{r}(\mathbb{R}^{n}))}\leq
h^{-n(1-\frac{2}{r})}\sup_{t\in (0,T_{0}]}
t^{\frac{2}{q}}\gamma(\frac{t}{h})^{1-\frac{2}{r}}\leq
Ch^{-2\beta}\Big(\sup_{s\in (0,\frac{T_{0}}{h}]}
s^{\alpha}\gamma(s))\Big)^{1-\frac{2}{r}}.
\end{equation}
\end{proof}

\subsection{Propagation of positivity}\label{propapos}
On $\mathbb{C}^{2m}=\mathbb{C}^{m}_{z}\times\mathbb{C}^{m}_{\zeta}$ one considers the symplectic $2$-form $\sigma=dz\wedge d\zeta=:\sigma_{\mathbb{R}}+i\sigma_{\mathbb{I}}$.  
\begin{dfn}\label{dfnlagrang}
Let $\Lambda$ be a smooth manifold of $\mathbb{C}^{2m}$. It is called
\begin{enumerate}
\item  $\mathbb{R}$ (resp. $\mathbb{I}$, $\mathbb{C}$)-Lagrangian if its dimension on $\mathbb{R}$ is $2m$ and $\sigma_{\mathbb{R}}|_{\Lambda}=0$ (resp. $\sigma_{\mathbb{I}}|_{\Lambda}=0$, $\sigma_{\mathbb{C}}|_{\Lambda}=0$); \item $\mathbb{R}$ (resp. $\mathbb{I}$)-symplectic if $\sigma_{\mathbb{R}}|_{T\Lambda}$ (resp. $\sigma_{\mathbb{I}}|_{T\Lambda}$) is nondegenerate; \item positive at some point $\rho\in\Lambda$ if the (real-valued) quadratic form $Q:u\rightarrow\frac{1}{i}\sigma(u,\bar{u})$ is positive definite on the tangent space $T_{\rho}\Lambda$ of $\Lambda$ at $\rho$.
\end{enumerate}
\end{dfn}
\begin{lemma}
Assume that the projection $\Lambda\ni(z,\zeta)\rightarrow z\in\mathbb{C}^{m}$ is a local diffeomorphism. Then $\Lambda$ is a $\mathbb{C}$-Lagrangian if and only if it is locally described by an equation of the type $\zeta=\frac{\partial\Phi}{\partial z}$, where $\Phi$ is a holomorphic function of $z$ and we write (locally) $\Lambda=\Lambda_{\Phi}$.
\end{lemma}
\begin{lemma}
A $\mathbb{C}$-Lagrangian $\Lambda$ is positive at some point $\rho$ if and only if near $\rho$ it is of the form $\Lambda_{\Phi}$, where $\Phi$ is a holomorphic function such that the real symmetric matrix $(Im\frac{\partial^{2}\Phi}{\partial z_{j}\partial z_{k}})_{j,k=\overline{1,m}}$ is positive definite.
\end{lemma}
Let $q=q(z,\zeta)$ be a holomorphic function on an open subset $U\subset\mathbb{C}^{2m}$. Then as in the real domain one defines the Hamilton field of $q$ by the identity $\sigma(u,H_{q}(z,\zeta))=dq(z,\zeta)u$. One also defines the Hamilton flow $\exp sH_{q}(z,\zeta)$ for $s$ real, by 
\begin{equation}
\frac{\partial}{\partial s}\exp sH_{q}(z,\zeta)=H_{q}(\exp sH_{q}(z,\zeta)) 
\end{equation}
and one can easily prove that for any open subset $U'\subset\subset U$ and for any $s\in\mathbb{R}$ such that $\cup_{s'\in[0,s]}\exp s'H_{q}(U')\subset U$, the application $U'\ni(z,\zeta)\rightarrow\exp sH_{q}(z,\zeta)$ is a complex canonical transformation.
\begin{lemma}\label{lemlag}
Let $\Lambda$ be a $\mathbb{C}$-Lagrangian submanifold of $\mathbb{C}^{2m}$ and assume that there exists $\rho\in\Lambda\cap\mathbb{R}^{2m}$ such that $\Lambda$ is positive at $\rho$. Moreover assume that there exists a complex canonical transformation $\kappa$ defined on a complex domain containing $\mathbb{R}^{2m}$ such that $\kappa(\mathbb{R}^{2m})\subset\mathbb{R}^{2m}$ and $\kappa(\rho)\in\Lambda$. Then $\Lambda$ is positive at $\kappa(\rho)$.
\end{lemma}
\begin{proof}
Observe that if $u\in T_{\kappa(\rho)}\Lambda$, then $u=d\kappa(\rho)v$ with $v\in T_{\rho}\Lambda$ and $\bar{u}=\overline{d\kappa(\rho)}\bar{v}=d\kappa(\bar{\rho})\bar{v}=d\kappa{\rho}\bar{v}$. Take now $\kappa=\exp sH_{q}$. For the proofs see \cite{mart02}, \cite{sjos}.
\end{proof}

\subsection{Airy functions}\label{secairy}
We give below some of the basic properties of the function $Ai(z)$ which are used in this work. For $z\in\mathbb{R}$, $Ai(z)$ is defined by
\begin{equation}
Ai(z)=\frac{1}{2\pi}\int_{-\infty}^{\infty} e^{i(\frac{u^{3}}{3}+zu)}du=\frac{1}{2\pi}\int_{-\infty}^{\infty} \cos(\frac{u^{3}}{3}+zu)du.
\end{equation}
This integral is not absolutely convergent, but is well defined as the Fourier transform of a temperate distribution.
For pozitive $z>0$, $z\rightarrow\infty$ we have
\begin{equation}\label{ai3}
Ai(z)=O(z^{-\infty}),
\end{equation}
\begin{equation}\label{air}
Ai(-z)=A^{+}(-z)+A^{-}(-z)(\simeq \frac{1}{\sqrt{\pi}}z^{-\frac{1}{4}}\cos(\frac{2}{3}z^{3/2}-\frac{\pi}{4})),
\end{equation}
where
\begin{equation}\label{simai}
A^{\pm}(-z)\simeq z^{-1/4}e^{\mp\frac{2i}{3}z^{3/2}\pm\frac{i\pi}{2}-\frac{i\pi}{4}}(\sum_{j=0}^{\infty}a_{\pm,j}(-1)^{-j/2}z^{-3j/2}),\quad a_{\pm,0}=\frac{1}{4\pi^{3/2}}.
\end{equation}
\begin{prop}\label{propzeros}
All the zeroes of $Ai(z)$ are real and negative, say
\begin{equation}
Ai(-\omega_{j})=0,\quad 0>-\omega_{0}>-\omega_{1}>...\rightarrow -\infty.
\end{equation}
\end{prop}

\subsubsection{Proof of Lemma \ref{lem2}}
\begin{proof}
Let $k\geq 0$ be fixed. For $x>0$ write
\begin{equation}
\psi(hD_{y})u(x,y)=\frac{1}{(2\pi h )^{d-1}}\int e^{iy\eta/h}
Ai(xh^{-2/3}|\eta|^{2/3}-\omega_{k})\psi(\eta)\hat{\varphi}(\frac{\eta}{h})d\eta.
\end{equation}
The change of variables $x=h^{2/3}\zeta$ reduces the proof to
the verification of the following inequality
\begin{equation}\label{ineg}
\|\psi_{1}(hD_{y})\varphi\|_{L^{r}(\mathbb{R}^{d-1})}\lesssim
\|\psi(hD_{y})u(h^{2/3}\zeta,.)\|_{L^{r}(\mathbb{R}_{+}\times\mathbb{R}^{d-1})}\lesssim\|\psi_{2}(hD_{y})\varphi\|_{L^{r}(\mathbb{R}^{d-1})}.
\end{equation}
Since $\psi$ is (compactly) supported away from $0$, let $\text{supp} (\psi)\subset \{0<|\eta_{0}|\leq|\eta|\leq|\eta_{1}|\}$. For $j\in\{0,1\}$, let $\epsilon_{j}>0$ be fixed and set $\zeta_{0}=(\omega_{k}-\omega_{0}+\epsilon_{0})|\eta_{0}|^{-2/3}$, $\zeta_{1}=(\omega_{k}+1+\epsilon_{1})|\eta_{1}|^{-2/3}$.
\begin{itemize}
\item For $\zeta\in [\zeta_{0},\zeta_{1}]$ we have, by Proposition \ref{propzeros}, 
\[
-\omega_{0}<-\omega_{0}+\epsilon_{0}=\zeta_{0}|\eta_{0}|^{2/3}-\omega_{k}\leq z=\zeta|\eta|^{2/3}-\omega_{k}\leq \zeta_{1}|\eta_{1}|^{2/3}-\omega_{k}=1+\epsilon_{1}.
\] 
For these values of the argument $z\in[-\omega_{0}+\epsilon_{0},1+\epsilon_{1}]$, $Ai(z)$ is positive, bounded from above and below which immediately yields, together with the assumption $\psi_{1}=\psi\psi_{1}$
\begin{equation}\label{esta1}
\|\psi_{1}(hD_{y})\psi(hD_{y})\varphi\|_{L^{r}(\mathbb{R}^{d-1})}\leq C_{1}
\|\psi(hD_{y})u(h^{\frac{2}{3}}\zeta,.)\|_{L^{r}([\zeta_{0},\zeta_{1}]\times\mathbb{R}^{d-1})},
\end{equation}
\[
C_{1}=\frac{\sup_{\eta} |\psi_{1}(\eta)|}{\inf_{z\in[-\omega_{0}+\epsilon_{0},1+\epsilon_{1}]}|Ai(z)|},
\]
consequently
\begin{equation}
\|\psi_{1}(hD_{y})\varphi\|_{L^{r}(\mathbb{R}^{d-1})}\leq C_{1}
\|\psi(hD_{y})u(h^{\frac{2}{3}}\zeta,.)\|_{L^{r}((0,\infty)\times\mathbb{R}^{d-1})}.
\end{equation}
\item On the other hand, since $Ai(z)$ is bounded for $z\in\mathbb{R}$, we obtain, since $\psi=\psi\psi_{2}$
\begin{equation}\label{esta2}
\|\psi(hD_{y})u(h^{\frac{2}{3}}\zeta,.)\|_{L^{r}((0,\infty)\times\mathbb{R}^{d-1})}\leq C_{2}\|\psi_{2}(hD_{y})\varphi\|_{L^{r}(\mathbb{R}^{d-1})},
\end{equation}
where 
\[
C_{2}=\sup_{\eta}|\psi(\eta)|\sup_{z}|Ai(z)|.
\] 
\end{itemize}
\end{proof}

\subsection{$L^{r}$ norms of the phase integrals associated to a cusp type
Lagrangian}\label{lrnorm}
\begin{prop}\label{propnorm}
For $t\in [4na^{1/2}-2a^{1/2}(1+c_{0}),4na^{1/2}+2a^{1/2}(1+c_{0})]$, the $L^{r}(\Omega)$ norm of a cusp $u^{n}_{h}(t,.)$ of the form \eqref{integral} are estimated (uniformly in $t$) by
\begin{equation}\label{estnorm}
\|u^{n}_{h}(t,.)\|_{L^{r}(\Omega)}\simeq
\left\{
                \begin{array}{ll}
                h^{\frac{1}{r}+\frac{1}{2}}a^{\frac{1}{r}-\frac{1}{4}}, \quad 2\leq r<4,\\
                h^{\frac{1}{3}+\frac{5}{3r}}, \quad r>4.\\
                \end{array}
                \right.
\end{equation}
\end{prop}
From Proposition \ref{propnorm} we deduce the following
\begin{cor}\label{corn}
For $t\in [4na^{1/2}-2a^{1/2}(1+c_{0}),4na^{1/2}+2a^{1/2}(1+c_{0})]$, the $L^{r}(\Omega)$ norms of a cusp $u^{n}_{h}(t,.)$ satisfy
\begin{itemize}
\item 
for $2\leq r<4$ 
\begin{equation}\label{estimnorm111}
\|u^{n}_{h}(t,.)\|_{L^{r}(\Omega)}\simeq
h^{\frac{1}{r}+\frac{1}{2}+\delta(\frac{1}{r}-\frac{1}{4})},
\end{equation}
\begin{equation}\label{estnorm211}
\|u^{n}_{h}(0,.)\|_{L^{2}(\Omega)}\simeq
h^{1+\frac{\delta}{4}}.
\end{equation}
\item for $r>4$ 
\begin{equation}\label{estnorm3}
\|u^{n}_{h}(t,.)\|_{L^{r}(\Omega)}\simeq
h^{\frac{1}{3}+\frac{5}{3r}}.
\end{equation}
\end{itemize}
\end{cor}

\begin{proof}
Let $0\leq n\leq N\simeq\lambda h^{\epsilon}$ be fixed and let 
\[
t\in [4na^{1/2}-2a^{1/2}(1+c_{0}),4na^{1/2}+2a^{1/2}(1+c_{0})].
\]
We compute the $L^{r}(\Omega)$ norms of
\begin{multline}
u^{n}_{h}(t,x,y)=\int_{\mathbb{R}^{2}}e^{\frac{i\eta}{h}(y-(1+a)^{1/2}t+(x-a)s+s^{3}/3-\frac{4}{3}na^{3/2})}\times\\ \times \Psi(\eta)\varrho^{n}(\frac{t+2(1+a)^{1/2}s}{2(1+a)^{1/2}a^{1/2}}-2n,\eta,\lambda)ds
d\eta,
\end{multline}
where the symbol $\varrho^{n}(.,\eta,\lambda)\in \mathcal{S}_{[-c_{0},c_{0}]}(\lambda/(n+1))$ defined in \eqref{dfnmodvrn} is essentially supported for the first variable in $[-c_{0},c_{0}]$ and where $\eta$ close to $1$ on the support of $\Psi$. Notice that due to the translation $y\rightarrow (y-t(1+a)^{1/2}+\frac{4}{3}na^{3/2})$ and the change
of variable $x\rightarrow (a-x)$ we are reduced to estimate the norm of 
\[
v^{n}_{h}(z,x,y):=\int e^{\frac{i\eta}{h}(y+\frac{s^{3}}{3}-sx)}\varrho^{n}(z+\frac{s}{h^{\frac{\delta}{2}}},\eta,\lambda)\Psi(\eta)dsd\eta,
\]
%where $\varrho^{n}(.)\in\mathcal{S}_{[-c_{0},c_{0}]}(\lambda/(n+1))$ for $\lambda=a^{3/2}/h\gg 1$
for $z=\frac{t}{2(1+a)^{1/2}a^{1/2}}-2n\in[-(1+c_{0}),1+c_{0}]$.
%\begin{rmq}
%For $0\leq n\leq N\simeq\lambda h^{\epsilon}$ the $L^{r}$ norms are estimated in the same way making translations in the tangential variable $y$ and using the fact that $\varrho^{n}(z,\eta,\lambda)$ writes as a convolution product $\varrho^{n}(z,\eta,\lambda)=(F_{\eta\lambda})^{*n}*\varrho^{0}(.,\lambda)(z)$.
%\end{rmq}
We distinguish several regions:
\begin{itemize}
\item For $|x|\leq Mh^{2/3}$
where $M$ is a constant, we make the changes of variables $x=\zeta
h^{2/3}$ and $s=h^{1/3}u$ which gives
\begin{equation}\label{umic}
I(z,x,\eta,h):=\int
e^{\frac{i\eta}{h}(\frac{s^{3}}{3}-sx)}\varrho^{n}(z+h^{-\delta/2}s,\eta,\lambda)ds
\end{equation}
\[
=h^{1/3}\int
e^{i\eta(\frac{u^{3}}{3}-u\zeta)}\varrho^{n}(z+h^{1/3-\delta/2}u,\eta,\lambda)
du.
\]
Let $Q(\zeta,u)=\frac{u^{3}}{3}-\zeta u$ and for $\theta:\mathbb{R}\rightarrow [0,1]$, set
\begin{equation}
F_{\theta}(w,\zeta,z,h)=\int e^{iw\eta}\Psi(\eta)f_{\theta}(\zeta,\eta,z,h)d\eta,
\end{equation}
\begin{equation}
f_{\theta}(\zeta,\eta,z,h)=\int e^{i\eta
Q(\zeta,u)}\theta(u)\varrho^{n}(z+h^{1/3-\delta/2}u,\eta,\lambda)du.
\end{equation}
We make integrations by parts in order to compute
\[
w^{k}F_{\theta}(w,\zeta,z,h)=i^{k}\int e^{iw\eta}\partial^{k}_{\eta}(\Psi(\eta)f_{\theta}(\zeta,\eta,z,h))d\eta,
\]
\[
\partial^{k}_{\eta}f_{\theta}(\zeta,\eta,z,h)=\int e^{i\eta Q(\zeta,u)}\theta(u)\sum_{j=0}^{k}C_{k}^{j}(iQ)^{k-j}\partial^{j}_{\eta}\varrho^{n}(z+h^{1/3-\delta/2}u,\eta,\lambda) du.
\]
Let $\theta(u)=1_{|u|\leq\sqrt{1+M}}$. Since we integrate for $\eta$ in a neighborhood of $1$, for all $k\geq 0$ we estimate
\begin{equation}\label{bau}
\|w^{k}F_{\theta}(w,\zeta,z,h)\|_{L^{\infty}_{w}}\leq 
\end{equation}
\[
\sum_{j=0}^{k}C_{k}^{j}\sup_{|u|\leq\sqrt{1+M}}|Q(\zeta,u)|^{k-j}\int |\partial^{j}_{\eta}\varrho^{n}(z+h^{1/3-\delta/2}u,\eta,\lambda)|d\eta \leq C_{k,M},
\]
where $C_{k,M}$ are constants and where we used the fact that $\varrho^{n}$ writes as a convolution product $\varrho^{n}(z,\eta,\lambda)=(F_{\eta\lambda})^{*n}*\varrho^{0}(.,\lambda)(z)$ and the derivatives in $\eta$ of $(F_{\eta\lambda})^{*n}$ were computed in Lemma \ref{lemlagramodwfset}.
\bigskip

For $\sqrt{1+M}\leq |u|\lesssim h^{\frac{\delta}{2}-\frac{1}{3}}$ we
integrate by parts using the operator
$L=\frac{\partial_{u}}{i\eta
\partial_{u}Q}$ which satisfies $L(e^{i\eta Q})=e^{i\eta Q}$. If we denote, for fixed $k$, $j\in\{0,..,k\}$ 
\[
Q^{k,j}_{0}:=(1-\theta(u))Q^{k-j}\partial^{j}_{\eta}\varrho^{n}(z+h^{1/3-\delta/2}u,\eta,\lambda),
\] 
and for $l\geq 0$, $Q^{k,j}_{l+1}=
\partial_{u}(\frac{Q^{k,j}_{l}}{\partial_{u}Q})$, then we can write
\begin{equation}
\int L^{l}(e^{i\eta Q})(1-\theta(u))Q^{k,j}_{0}du=\frac{(-1)^{l}}{(i\eta)^{l}}
\int e^{i\eta Q}Q^{k,j}_{l}du,
\end{equation}
where 
\[
Q^{k,j}_{l}=\sum_{m=0}^{l}c^{k,j}_{l,m}(\zeta,\varrho^{n}(z+.,\eta,\lambda))|u|^{3(k-j)-3l+m}h^{m
(\frac{1}{3}-\frac{\delta}{2})},
\]
where $c^{k,j}_{l,m}$ depends on the derivatives $\partial^{l-m}_{u}\partial^{j}_{\eta}\varrho^{n}(z+.,\eta,\lambda)$. The principal term is obtained for $j=0$ and $m=0$ and it equals $|u|^{3k-3l}$. It's enough to take $l=2k$ to obtain similar bounds for $\|w^{k}F_{1-\theta}(w,\zeta,z,h)\|_{L^{\infty}_{w}}$ as in \eqref{bau}. We find
\begin{multline}\label{estnorm}
\|v^{n}_{h}(z,.)\|_{L^{r}(|x|\leq
Mh^{2/3},y)}=h^{2/3r}(\int_{y}\int_{0}^{M}|v^{n}_{h}(z,h^{2/3}\zeta,y)|^{r}d\zeta
dy)^{1/r}\\
=h^{5/3r+1/3}\|F_{1}(w,\zeta,z,h)\|_{L^{r}(\zeta\leq M,w)}\simeq
h^{5/3r+1/3}.
\end{multline}

\item For $x\in(Mh^{2/3},A]$ with $M\gg1$ big enough we apply the stationary phase theorem:
\begin{prop}\label{thmphasestat}(\cite[Thm.7.7.5]{hormand})
Let $K\subset\mathbb{R}$ be a compact set, $f\in C^{\infty}_{0}(K)$,
$\phi\in C^{\infty}(\mathring{K})$ such that $\phi(0)=\phi'(0)=0$,
$\phi''(0)\neq 0$, $\phi'\neq 0$ in
$\mathring{K}\setminus \{0\}$. Let $\omega\gg 1$, then for every
$k\geq 1$ we have
\begin{equation}
|\int e^{i\omega\phi(u)}f(u)du-\frac{(2\pi i
)^{\frac{1}{2}}e^{i\omega\phi(0)}}{(\omega\phi''(0))^{\frac{1}{2}}}\sum_{j<k}\omega^{-j}L_{j}f|\leq C\omega^{-k}\sum_{|\alpha|\leq 2k}\sup|\partial^{\alpha}f|.
\end{equation}
Here $C$ is bounded when $\phi$ stays in a bounded set in $C^{\infty}(\mathring{K})$, $|u|/|\phi'(u)|$ has a uniform bound and
\begin{equation}\label{hormand}
L_{j}f=\sum_{\nu-\mu=j}\sum_{2\nu\geq3\mu}\frac{i^{-j}2^{-\nu}}{\mu!\nu!}(\phi''(0))^{-\nu}\partial^{2\nu}
(\kappa^{\mu}f)(0).
\end{equation}
where
$\kappa(u)=\phi(u)-\phi(0)-\frac{\phi''(0)}{2}u^{2}$ vanishes of third order at $0$. 
\end{prop}
We make the change of variable $s=\sqrt{x}(\pm 1+u)$ to compute
the integral in $s$ in the expression of $v_{h}$. Using Proposition \ref{thmphasestat} with $\phi_{\pm}(u)=\frac{u^{3}}{3}\pm u^{2}$, $\omega=\eta\frac{x^{3/2}}{h}\gg 1$, $\kappa_{\pm}(u)=u^{3}/3$ we write $I(z,x,\eta,h)$ as a sum $I(z,x,\eta,h)\simeq\sum_{\pm,j\geq 0}I^{j}_{\pm}(z,x,\eta,h)$, where
\begin{multline}\label{ispm}
I^{j}_{\pm}(z,x,\eta,h):=
(i\pi)^{1/2}h^{1/2+j}\eta^{-1/2-j}e^{\mp\frac{2}{3}i\eta x^{3/2}/h} x^{-1/4-3j/2}\times \\ \times L_{j}(\varrho^{n}(z+h^{-\frac{\delta}{2}}\sqrt{x}(\pm1+u),\eta,\lambda))|_{u=0}.
\end{multline}
We compute each $L^{r}$ norm of $\int e^{\frac{i y\eta}{h}}\Psi(\eta)I^{j}_{\pm}(z,x,\eta,h)d\eta$:
\begin{multline}\label{normtoest}
\|\int e^{\frac{i y\eta}{h}}\Psi(\eta)I^{j}_{\pm}(z,x,\eta,h)d\eta\|_{L^{r}(x\in(Mh^{2/3},A],y)}\simeq\\
h^{1/2+j}\|x^{-1/4-3j/2}
\int e^{\frac{i\eta}{h}(y\mp\frac{2}{3} x^{3/2})}\Psi(\eta)
\eta^{-1/2-j}L_{j}(\varrho^{n}(z\pm h^{-\delta/2}x^{1/2},\eta,\lambda))d\eta\|_{L^{r}(x\in(Mh^{2/3},A],y)}.
\end{multline}
Using again the fact that $\varrho^{n}(z,\eta,\lambda)=(F_{\eta\lambda})^{*n}*\varrho^{0}(.,\lambda)(z)$ we introduce the map $F^{n,j}(z,\eta):=\Psi(\eta)\eta^{-1/2-j}(F_{\eta\lambda})^{*n}(z)$ which is compactly supported in $\eta$; if $\widehat{F^{n,j}}(z,.)$ denotes its Fourier transform with respect to $\eta$, \eqref{normtoest} reads
\[
h^{1/2+j}\|x^{-1/4-3j/2}
\widehat{F^{n,j}}(.,\frac{(y\mp\frac{2}{3} x^{3/2})}{h})*L_{j}(\varrho^{n}(.,\eta,\lambda)(z\pm h^{-\delta/2}x^{1/2}))\|_{L^{r}(x\in(Mh^{2/3},A],y)}.
\]
Setting $y=hw$, $x=h^{2/3}\zeta$ and translating $w\rightarrow w\mp\frac{2}{3}\zeta^{3/2}$ we can estimate from above and from below each one of the above norms.
For $j\geq 0$, $L_{j}$ is a differential operator of order $2j$ and each derivative on $\varrho$ gives a factor $\sqrt{x}/h^{\delta/2}\leq 1$. We estimate the $L^{r}$ norm of $\int e^{\frac{i y\eta}{h}}\Psi(\eta)I(z,x,\eta,h)d\eta$ from above and from below by the sum over $j$ of
\[
Ch^{r(1/2+j+5/3r-1/6-j)}\int_{M}^{Ah^{-2/3}}\zeta^{-r(1/4+3j/2)}d\zeta
\]
where $C>0$ are constants, and since the operators $L_{j}$ are of order $2j$, for each $j$ there will be $2j$ terms in the sum: summing up over $j\geq 0$ (taking $M\geq 2$ for example) and using the assumption $\varrho^{n}\in\mathcal{S}_{[-c_{0},c_{0}]}(\lambda/(n+1))$ which assures uniform bounds for the derivatives $\partial^{j}\varrho^{n}(.,\eta,\lambda)$ for each $n,j\geq 0$, we obtain for $r>4$
\[
\|\int e^{\frac{i y\eta}{h}}\Psi(\eta)I(z,x,\eta,h)d\eta\|^{r}_{L^{r}(x\in(Mh^{2/3},A],y)}\simeq h^{r/3+5/3}\sum_{j\geq 0}\frac{j M^{1-r(1/4+3j/2)}}{(r(1/4+3j/2)-1)},
\]
and taking $M\geq 2$ sufficiently big we can sum over $j$ and we deduce \eqref{estnorm} for $r>4$.

For $r\in [2,4)$ and $j=0$ we have $r/4-1<0$ and 
\[
h^{r(1/2+5/3r-1/6)}\int_{M}^{Ah^{-2/3}}\zeta^{-r/4}d\zeta\simeq \frac{(Ah^{-2/3})^{1-r/4}}{1-r/4}.
\]
For $r\in [2,4)$ and $j\geq 1$ we have $r(1/4+3j/2)-1>0$ and 
\[
h^{r(1/2+j+5/3r-1/6-j)}\int_{M}^{Ah^{-2/3}}\zeta^{-r(1/4+3j/2)}d\zeta\simeq \frac{M^{1-r(1/4+3j/2)}}{r(1/4+3j/2)-1}.
\]
If $M\geq 2$ is sufficiently large the sum of the $L^{r}$ norms over $j\geq 0$ is small compared to the norm for $j=0$, hence \eqref{estnorm} follows for $r\in [2,4)$ too.

\item In the last case $x>a$ the $L^{r}(\Omega)$ norms are as small as we want since the contribution of $u^{n}_{h}$ in this case is $O_{L^{2}}(h^{\infty})$, because this region is localized away from a neighborhood of the Lagrangian $\Lambda_{\Phi_{n}}$ and we use Lemma \ref{lemlagramodwfset}. 
\end{itemize}
\end{proof}


\nocite{*}
\begin{thebibliography}{hormand}

\bibitem{blsmso08} M.Blair, H.Smith, C.Sogge:
 Strichartz estimates for the wave equation on manifolds with boundary, arxiv:0805.4733v2


\bibitem{bulepl07} N.Burq, G.Lebeau, F.Planchon: Global existence for energy critical waves in $3D$ domains, {J.Amer.Math.Soc.} 21, 831-845 (2008)

\bibitem{esk77}
G.~Eskin: Parametrix and propagation of singularities for the interior mixed hyperbolic problem, {J}.{A}nalyse {M}ath., 32, 17-62 (1977)

\bibitem{give85}
J.~Ginibre, G.~Velo: The global {C}auchy problem for the critical nonlinear {S}chr$\ddot{o}$dinger equation in $H^{s}$, {A}nn.{I}.{H}.{P}.{A}nal.non-lin.2, 309-327 (1985)

\bibitem{give95}
J.~Ginibre, G.~Velo: Generalized {S}richartz inequalities for the wave equation, {J}.{F}unct.{A}nal.133, 50-68 (1995)

\bibitem{hormand}
L.~H\"ormander: The analysis of linear partial differential operators {I},{III} , Springer-{V}erlag (1985)

\bibitem{doi} 
O.~Ivanovici: Counter examples to {S}trichartz estimates for the wave equation in domains {II}, to be submitted

\bibitem{lev90}
L.~Kapitanski: Some generalizations of the {S}trichartz-{B}renner inequality, {L}eningrad {M}ath.{J}.1, 693-676 (1990)

\bibitem{kosmta}
H.~Koch, H.~Smith, D.~Tataru: Subcritical $L^{p}$ bounds on spectral clusters for Lipschitz metrics, {Math.Res.Lett.} 15, no.5 (2008)

\bibitem{kt98}
M.~Keel, T.~Tao: Endpoints {S}trichartz estimates, {A}mer.{J}.{M}ath. 120, 955-980 (1998)

%\bibitem{gle}
%G.~Lebeau,
%\newblock{\em Cours \'ecole d'\'et\'e {N}ice}, September 2006

\bibitem{gle06}
G.~Lebeau: Estimations dispersives pour les ondes dans un domaine strictement convexe,  \url{http://archive.nundam.org/item?id=JEDP_2006___A7_0}, {E}vian (2006)

\bibitem{ls95}
H.~Lindblad, C.~Sogge: On existence and scattering with minimal regularity for semilinear wave equations, {J}.{F}unct.{A}nal.130, 357-426 (1995)

\bibitem{mart02}
A.~Martinez: An {I}ntroduction to {S}emiclassical and {M}icrolocal {A}nalysis, Springer (2002)

\bibitem{sjos}
J.~Sjostrand: Singularites analytiques microlocales, {A}sterisque 95 (1982)

\bibitem{smso95}
H.~Smith, C.~Sogge: On the critical semilinear wave equation outside convex obstacles, {J}.{A}mer.{M}ath.{S}oc. 8, 897-916 (1995)

\bibitem{smso06}
H.~Smith, C.~Sogge: On the $L^{p}$ norm of spectral clusters for compact manifolds with boundary, {A}cta {M}atematica 198, no.1 (2007)

\bibitem{stri77}
R.S.~Strichartz: Restriction of {F}ourier transforms to quadratic surfaces and decay of solutions of wave equation, {D}uke.{M}ath.{J}. 44, 705-714 (1977)

\end{thebibliography}
\end{document}